\documentclass[a4paper, reqno, 10pt]{amsart}

\usepackage[utf8]{inputenc}
\usepackage{amsthm, amsfonts, amssymb, amsmath}
\usepackage{enumitem}
\usepackage{eqparbox}
\usepackage{etoolbox}
\usepackage{pdfsync}
\usepackage{extpfeil}
\usepackage{stmaryrd}
\usepackage{mathtools}

\usepackage{mathrsfs}
\usepackage{bbm}
\usepackage{tikz-cd}
\usepackage[colorlinks=False, citecolor=Black, linkcolor=Red, urlcolor=Blue]{hyperref}
\usepackage{cleveref}
\usepackage[backend=biber, style=alphabetic, giveninits]{biblatex}
\DeclareSourcemap{
	\maps[datatype=bibtex]{
		\map{
			\step[fieldsource=doi, final]
			\step[fieldset=isbn, null]
			\step[fieldset=issn, null]
			\step[fieldset = url, null]
		}
	}
}
\DeclareFieldFormat{postnote}{#1}
\DeclareFieldFormat{multipostnote}{#1}

\DeclareFieldFormat
[article,inbook,incollection,inproceedings,patent,thesis,unpublished]
{title}{\mkbibemph{#1\isdot}}

\DeclareFieldFormat
[article,inbook,incollection,inproceedings,patent,thesis,unpublished]
{volume}{\mkbibbold{#1\isdot}}

\newtheorem{theoremalphabetical}{Theorem}[section]

\newtheorem{theorem}{Theorem}[section]
\newtheorem{proposition}[theorem]{Proposition}
\newtheorem{lemma}[theorem]{Lemma}

\newtheorem{corollary}[theorem]{Corollary}

\theoremstyle{definition}
\newtheorem{definition}[theorem]{Definition}

\newtheorem{example}[theorem]{Example}

\newtheorem{remark}[theorem]{Remark}

\newenvironment{altenumerate}
{\begin{list}
		{\textup{(\theenumi)} }
		{\usecounter{enumi}
			\setlength{\labelwidth}{0pt}
			\setlength{\labelsep}{0pt}
			\setlength{\leftmargin}{0pt}
			\setlength{\itemsep}{0pt}
			\setlength{\topsep}{0pt}
			\renewcommand{\theenumi}{\roman{enumi}}
	}}
	{\end{list}}

\newenvironment{altenumeratelevel2}
{\begin{list}
		{\textup{(\theenumi)} }
		{\usecounter{enumii}
			\setlength{\labelwidth}{2em}
			\setlength{\labelsep}{0pt}
			\setlength{\leftmargin}{2em}
			\setlength{\itemsep}{2pt}
			\setlength{\topsep}{2pt}
			\setlength{\itemindent}{0pt}
			\renewcommand{\theenumi}{\arabic{enumii}}
	}}
	{\end{list}}

\newcommand{\eqmathbox}[2][N]{\eqmakebox[#1]{$\displaystyle#2$}}

\numberwithin{equation}{section}

\makeatletter
\def\@seccntformat#1{%
	\protect\textup{\protect\@secnumfont
		\ifnum\pdfstrcmp{subsection}{#1}=0 \bfseries\fi% subsection # in \bfseries
		\csname the#1\endcsname
		\protect\@secnumpunct
	}%
}  
\makeatother

\makeatletter
\def\@tocline#1#2#3#4#5#6#7{\relax
	\ifnum #1>\c@tocdepth % then omit
	\else
	\par \addpenalty\@secpenalty\addvspace{#2}%
	\begingroup \hyphenpenalty\@M
	\@ifempty{#4}{%
		\@tempdima\csname r@tocindent\number#1\endcsname\relax
	}{%
		\@tempdima#4\relax
	}%
	\parindent\z@ \leftskip#3\relax \advance\leftskip\@tempdima\relax
	\rightskip\@pnumwidth plus4em \parfillskip-\@pnumwidth
	#5\leavevmode\hskip-\@tempdima
	\ifcase #1
	\or\or \hskip 1em \or \hskip 2em \else \hskip 3em \fi%
	#6\nobreak\relax
	\hfill\hbox to\@pnumwidth{\@tocpagenum{#7}}\par% <---- \dotfill -> \hfill
	\nobreak
	\endgroup
	\fi}
\makeatother

\newcommand{\Ad}{{\mathrm{Ad}}}

\newcommand{\abs}[1]{{\lvert{#1}\rvert}}
\newcommand{\an}{{\mathrm{an}}}

\newcommand{\alg}{{\mathrm{alg}}}
\newcommand{\adm}{{\mathrm{adm}}}

\newcommand{\blank}{{\,\_\,}}

\renewcommand{\cot}[1]{{\,\widehat{\otimes}_{#1}\,}}

\renewcommand{\dim}{{\mathrm{dim}}}
\newcommand{\defeq}{\vcentcolon=}

\newcommand{\ev}{{\mathrm{ev}}}

\newcommand{\GL}{{\mathrm{GL}}}

\newcommand{\hy}{{\mathrm{hy}}}

\renewcommand{\Im}{{\mathrm{Im}}}

\newcommand{\id}{{\mathrm{id}}}
\newcommand{\Ind}{{\mathrm{Ind}}}
\newcommand{\inv}{{\mathrm{inv}}}

\newcommand{\Ker}{{\mathrm{Ker}}}

\newcommand{\la}{{\mathrm{la}}}

\newcommand{\lra}{\longrightarrow}

\newcommand{\lto}{\longmapsto}

\DeclareRobustCommand\longtwoheadrightarrow
{\relbar\joinrel\twoheadrightarrow}
\DeclareRobustCommand\longhookrightarrow
{\lhook\joinrel\longrightarrow}

\newcommand{\mto}{\mapsto}

\newcommand{\ov}[1]{\overline{#1}}
\newcommand{\ot}[1]{{\,\otimes_{#1}\,}}

\newcommand{\ra}{\rightarrow}

\newcommand{\res}[1]{{\!\,\mid_{#1}}}
\newcommand{\rig}{{\mathrm{rig}}}
\newcommand{\Rep}{{\mathrm{Rep}}}

\newcommand{\sm}{{\infty}}

\newcommand{\transp}[1]{{{#1}^\mathrm{t}}}

\newcommand{\ul}[1]{{\underline{#1}}}

\let\originalmiddle\middle
\renewcommand{\middle}[1]{\,\originalmiddle#1\,}

\newcommand{\llrrbracket}[1]{\mkern-3mu\left[\mkern-2.5mu\left[#1\right]\mkern-2.5mu\right]}

\newcommand{\BN}{{\mathbb {N}}}

\newcommand{\BQ}{{\mathbb {Q}}}

\newcommand{\CA}{{\mathcal {A}}}

\newcommand{\CC}{{\mathcal {C}}}

\newcommand{\CF}{{\mathcal {F}}}

\newcommand{\CI}{{\mathcal {I}}}

\newcommand{\CL}{{\mathcal {L}}}
\newcommand{\CM}{{\mathcal {M}}}

\newcommand{\CO}{{\mathcal {O}}}

\newcommand{\bB}{{\mathrm{\bf B}}}

\newcommand{\bG}{{\mathrm{\bf G}}}

\newcommand{\bL}{{\mathrm{\bf L}}}

\newcommand{\bP}{{\mathrm{\bf P}}}

\newcommand{\bT}{{\mathrm{\bf T}}}
\newcommand{\bU}{{\mathrm{\bf U}}}

\newcommand{\bX}{{\mathrm{\bf X}}}

\newcommand{\Fa}{{\mathfrak {a}}}

\newcommand{\Fd}{{\mathfrak {d}}}

\newcommand{\Fg}{{\mathfrak {g}}}

\newcommand{\Fl}{{\mathfrak {l}}}
\newcommand{\Fm}{{\mathfrak {m}}}

\newcommand{\Fp}{{\mathfrak {p}}}

\newcommand{\Ft}{{\mathfrak {t}}}
\newcommand{\Fu}{{\mathfrak {u}}}

\newcommand{\Fx}{{\mathfrak {x}}}

\newcommand{\Fz}{{\mathfrak {z}}}

\newcommand{\FM}{{\mathfrak {M}}}

\newcommand{\FZ}{{\mathfrak {Z}}}

\title{The Functors $\CF^G_P$ over Local Fields of Positive Characteristic}
\author{Georg Linden}
\date{July 9, 2024}

\begin{document}

\begin{abstract}
	Let $G$ be a split connected reductive group over a non-archimedean local field. 
	In the $p$-adic setting, Orlik--Strauch constructed functors from the BGG category $\CO$ associated to the Lie algebra of $G$ to the category of locally analytic representation of $G$. 
	We generalize these functors to such groups over non-archimedean local fields of arbitrary characteristic. 
	To this end, we introduce the hyperalgebra of a non-archimedean Lie group $G$, which generalizes its Lie algebra, and consider topological modules over the algebra of locally analytic distributions on $G$ and subalgebras related to this hyperalgebra.
\end{abstract}
\maketitle
\tableofcontents

\section{Introduction}

Let $L$ be a non-archimedean local field and let $K$ be a finite extension of $L$, which serves as our coefficient field.
Let $\bG$ be a split connected reductive algebraic group over $L$ and let $\bP \supset \bB \supset \bT$ be a parabolic subgroup, a Borel subgroup and a maximal split torus of $\bG$ respectively.
Moreover, let $G$, $P$, etc.\ denote the respective groups of $L$-valued points considered as locally $L$-analytic Lie groups.

When $L$ is a finite extension of $\BQ_p$, Orlik and Strauch \cite{OrlikStrauch15JordanHoelderSerLocAnRep} defined and studied bi-functors
\begin{equation*}
	\CF_P^G \colon \CO_\alg^\Fp \times \Rep^{\sm, \adm}_K (L_\bP) \lra \Rep^{\la,\adm}_K (G) ,
\end{equation*}
which give rise to admissible locally analytic $G$-representations via a procedure reminiscent of parabolic induction. 
Here $\bL_\bP$ is the standard Levi subgroup of $\bP$ and $\Fg$, $\Fp$, etc.\ denote the Lie algebras of the respective groups.
Moreover, $\CO_\alg^\Fp$ is the subcategory of the Bernstein--Gelfand--Gelfand category $\CO$ consisting of those finitely generated $\Fg$-modules which are locally $\Fp$-finite and on which $\Ft$ acts semisimply with algebraic weights.
In \textit{loc.\ cit.\ }Orlik and Strauch showed that the functors $\CF_P^G$ are exact and gave a criterion for when a representation in their image is irreducible.

One may also consider locally analytic $G$-representations defined in the same way when $L$ is of positive characteristic\footnote{The coefficient field $K$ continues to be a finite extension of $L$.}.
For instance, (the strong dual of) representations in the holomorphic discrete series of $\GL_n(L)$ are locally analytic, for any non-archimedean local field $L$\footnote{See \cite{Schneider92CohomLocSystemspAdicUnifVar} for the definition of and \cite{SchneiderTeitelbaum02pAdicBoundVal, Pohlkamp02RandwerteHolomFunkt, Orlik08EquivVBDrinfeldUpHalfSp} for results on holomorphic discrete series representations in the $p$-adic case.}.
%These can be realized as the rigid analytic sections of $\GL_{n,L}$-equivariant vector bundles on projective space restricted to the Drinfeld upper half space over $L$.
In this generality, Gräf \cite{Graef20BoundDistr} studied the connection between a certain family of holomorphic discrete series representations and harmonic cocycles on the Bruhat--Tits building for $\GL_3(L)$.
Moreover, in \cite{Linden23EquivVBDrinfeldUpHalfSp} we identified certain subquotients of holomorphic discrete series representations of $\GL_n(L)$ and described them using a construction similar to the functors $\CF_P^G$. 
This extended work of Orlik \cite{Orlik08EquivVBDrinfeldUpHalfSp} for the $p$-adic case.\\

The goal of the present paper is to systematically study a generalization of the functors $\CF_P^G$ to non-archimedean local fields of arbitrary characteristic. 
Our approach is based on topological modules over the topological algebra $D(G)$ of locally analytic distributions on $G$ and subalgebras of $D(G)$.

More precisely, let $\CM_{D(G)}$ be the category of \textit{separately continuous $D(G)$-modules} on locally convex $K$-vector spaces (\Cref{Def - Topological algebras and modules}), and let $\CM_{D(G)}^\mathrm{nF} $ denote its full subcategory consisting of modules on nuclear $K$-Fr\'echet spaces.
Then there is an anti-equivalence between $\CM_{D(G)}^\mathrm{nF} $ and the category $\Rep_K^{\la, \mathrm{ct}} (G)$ of locally analytic $G$-representations whose underlying locally convex $K$-vector space is of compact type. 
This anti-equivalence is realized by passing from a locally analytic representation $V$ to its strong dual $V'_b$ and vice versa, and it can be proven exactly as in the $p$-adic setting pioneered by Schneider and Teitelbaum \cite{SchneiderTeitelbaum02LocAnDistApplToGL2}.

We define the \emph{hyperalgebra} $\hy(G)$ of a non-archimedean Lie group $G$ as the subalgebra of $D(G)$ consisting of those distributions supported in the identity element $1 \in G$ which are of finite order (see \Cref{Def - Distributions with support}). 
When $L$ is $p$-adic, $\hy(G)$ coincides with the universal enveloping algebra $U(\Fg)$ of the Lie algebra of $G$.
However, for $\mathrm{char}(L)>0$ this hyperalgebra is better suited than $U(\Fg)$.
For example, the pairing between $\hy(G)$ and germs of locally analytic functions at $1$ is non-degenerate (\Cref{Cor - Pairing of hyperalgebra and vector valued germs}) -- a property that fails for $U(\Fg)$ in positive characteristic.

With $\hy(G)$ at our disposal we generalize a definition of Orlik--Strauch and consider the subalgebra $D(\dot{\Fg}, P)$ of $D(G)$ generated by $\hy(G)$ and $D(P)$.
It can also be expressed as the inductive tensor product
\begin{equation*}
	D(\dot{\Fg},P) \cong \hy(G) \ot{\hy(P),\iota} D(P) .
\end{equation*}
Here again, we have the categories $\CM_{D(\dot{\Fg},P)}$ (resp.\ $\CM_{D(\dot{\Fg},P)}^\mathrm{nF} $) of separately continuous $D(\dot{\Fg},P)$-modules (with underlying nuclear $K$-Fr\'echet space).
We show that $\CM_{D(\dot{\Fg},P)}^\mathrm{nF} $ is quasi-abelian and closed under taking completed tensor products (\Cref{Prop - Category of nuclear Frechet modules is quasi-abelian}, \Cref{Cor - Completed tensor product is module}).
We note that the aforementioned subquotients of holomorphic discrete series representations occurring in \cite{Linden23EquivVBDrinfeldUpHalfSp} lie in the category $\CM_{D(\dot{\Fg},P)}^\mathrm{nF} $, for certain parabolic subgroups $P$ of $ G =\GL_n(L)$.

\begin{theoremalphabetical}[{\Cref{Prop - Locally analytic representation via tensoring up}, \Cref{Thm - Functor via tensoring up is left exact}}]\label{Thm - A}
		\begin{altenumerate}
		\item 
			For $M \in \CM_{D(\dot{\Fg},P)}$ whose underlying locally convex $K$-vector space is pseudo-metrizable and nuclear,
			\begin{equation*}
				\dot{\CF}_P^G (M) \defeq \big( D(G) \cot{D(\dot{\Fg},P),\iota} M \big)'_b 
			\end{equation*}
			has a canonical structure of a locally analytic $G$-representation of compact type.
		
		\item 
		This association yields a left exact contravariant functor between quasi-abelian categories
		\begin{equation*}
			\dot{\CF}^G_P \colon \CM_{D(\dot{\Fg},P)}^\mathrm{nF} \lra \Rep_K^{\la, \mathrm{ct}} (G) \,,\quad M \lto \dot{\CF}^G_P(M) .
		\end{equation*}
	\end{altenumerate}
\end{theoremalphabetical}

We observe that for a locally analytic $P$-representation $V$ of compact type we naturally obtain a $D(\dot{\Fg},P)$-module structure on $\hy(G) \cot{\hy(P)} V'_b $ so that the latter lies in $\CM_{D(\dot{\Fg},P)}^\mathrm{nF}$ (\Cref{Cor - Composite module from H-representation}).
In this way, the functor $\dot{\CF}^G_P$ recovers locally analytic induction:
\begin{equation*}
	\dot{\CF}^G_P \big( \hy(G) \cot{\hy(P)} V'_b  \big) \cong \Ind_{P}^{\la, G}  ( V ) .
\end{equation*}

Our construction generalizes the Orlik--Strauch functors $\CF^G_P$ for $p$-adic $L$ in the following sense.
We equip each $M\in \CO_\alg^\Fp$ with a locally convex topology so that $M$ and its Hausdorff completion $\widehat{M}$ naturally become separately continuous $D(\dot{\Fg},P)$-modules.
This yields a strongly exact functor (\Cref{Cor - Composite module structure on completion of object in category O})
\begin{equation*}
	\CC \colon  \CO_\alg^\Fp \lra \CM_{D(\dot{\Fg},P)}^\mathrm{nF}\,,\quad M \lto \widehat{M}.
\end{equation*}
Furthermore, let $\Rep_K^{\sm, \mathrm{s\text{-}adm}} (P)$ denote the category of strongly admissible smooth $P$-re\-pre\-sen\-tations\footnote{
Here $V\in \Rep_K^{\infty}(P)$ is \textit{strongly admissible} if $V$ embeds as a subrepresentation into $C^\sm (P_0,K)^{\oplus n}$, for some compact open subgroup $P_0 \subset P$ and some $n\in \BN$, when viewed as a representation of $P_0$, see \cite[Sect.\ 2]{SchneiderTeitelbaumPrasad01UgFinLocAnRep}, \cite[Sect.\ 4.1.2]{AgrawalStrauch22FromCatOLocAnRep}.}.
For $V \in \Rep_K^{\sm, \mathrm{s\text{-}adm}} (P)$ we may form the inductive tensor product
\begin{equation}\label{Eq - Tensored up module}
	D(G) \ot{D(\dot{\Fg},P),\iota} \big(M \ot{K} V'_b \big) \tag{$*$}, 
\end{equation}
which is a separately continuous $D(G)$-module.
Agrawal and Strauch \cite[Thm.\ 4.2.3]{AgrawalStrauch22FromCatOLocAnRep} showed that the abstract $D(G)$-module underlying \eqref{Eq - Tensored up module} is coadmissible, and thus carries a canonical Fr\'echet topology.
We prove that this canonical Fr\'echet topology coincides with the above inductive tensor product topology (\Cref{Thm - Comparision between canonical Frechet topology and topological tensor product}).
This serves as an intermediate step to showing:

\begin{theoremalphabetical}[{\Cref{Thm - Dual of Orlik-Strauch functors}}]\label{Thm - B}
	For $M \in \CO_\alg^\Fp$ and $V\in \Rep_K^{\sm, \mathrm{s\text{-}adm}}(P)$, there is a canonical topological isomorphism of $D(G)$-modules
	\begin{equation*}
		 \CF^G_P (M,V)'_b \cong D(G) \ot{D(\dot{\Fg},P),\iota} \big(M \ot{K} V'_b \big) .
	\end{equation*}
	In particular, the functor $\CF_P^G$ factors over $\dot{\CF}_P^G$:
	\begin{equation*}
		\begin{tikzcd}
			\CO_\alg^\Fp \times \Rep_K^{\sm, \mathrm{s\text{-}adm}} (P) \ar[rrr, "\CF_P^G"] \ar[dr, "\CC \times (\blank)'_b"', end anchor= 164 ] &[-30pt] &[10pt] &[-30pt] \Rep_K^{\la,\adm}(G) \\
			&\CM_{D(\dot{\Fg},P)}^\mathrm{nF} \times \CM_{D(\dot{\Fg},P)}^\mathrm{nF} \ar[r, "\blank \, \widehat{\otimes}_{K} \blank"] & \CM_{D(\dot{\Fg},P)}^\mathrm{nF} \ar[ur, "\dot{\CF}_P^G"'] & .
		\end{tikzcd}
	\end{equation*}
\end{theoremalphabetical}

In \cite{AgrawalStrauch22FromCatOLocAnRep}, Agrawal and Strauch also adopt a strategy based on modules over locally analytic distribution algebras to extend the functors $\CF_P^G$ to the extension closure of the category $\CO_\alg^\Fp$  in the $p$-adic setting.
For them it often suffices to consider such modules entirely in algebraic terms thanks to Schneider--Teitelbaum's theory of coadmissible modules over Fr\'echet--Stein algebras \cite{SchneiderTeitelbaum03AlgpAdicDistAdmRep}.
We are not aware of an analogue of this theory applicable over local fields of positive characteristic\footnote{
For example, in general a compact non-archimedean Lie group $G$ over a local field of positive characteristic is not topologically finitely generated and thus in particular does not admit an open uniform pro-$p$ subgroup.
Hence a direct transfer of the reasoning in \cite{SchneiderTeitelbaum03AlgpAdicDistAdmRep} for $D(G)$ to be a Fr\'echet--Stein algebra fails.}, 
and thus take the topologies of the modules more into consideration.

\subsection{Structure of the Paper}

In \Cref{Sect - Preliminaries}, we collect some generalities about topological algebras and modules on locally convex $K$-vector spaces whose (scalar) multiplication is separately continuous.
In particular, for a continuous homomorphism $A \ra B$ of such algebras and a separately continuous $A$-module $M$, we consider the extension of scalars $B \ot{A,\iota} M$ and its completion $B \cot{A,\iota} M$ in the context of separately continuous $B$-modules (\Cref{Lemma - Extending separately continuous module structure to completion}, \Cref{Lemma - Base change for topological modules}).

\Cref{Sect - Locally analytic representation theory} is a recapitulation of foundational topics of locally analytic representation theory which are valid independently of the characteristic of $L$. 
This includes the definition and properties of locally analytic functions and distributions on locally $L$-analytic manifolds (\Cref{Prop - Properties of locally analytic functions}, \Cref{Prop - Properties of space of distributions}), the convolution product of $D(G)$ and the anti-equivalence between $\Rep_K^{\la,\mathrm{ct}}(G)$ and $\CM_{D(G)}^\mathrm{nF}$ (\Cref{Prop - Anti-equivalence for locally analytic representations}).
Finally, we consider locally analytically induced representations and, for $V \in \Rep_K^{\la,\mathrm{ct}}(P)$, prove in \Cref{Prop - Module description for induction} that
\begin{equation*}
	\big( \Ind_P^{\la,G} (V) \big)'_b \cong D(G) \cot{D(P),\iota} V'_b.
\end{equation*}

In \Cref{Sect - The hyperalgebra}, after recalling the notion of germs of locally analytic functions (resp.\ of locally analytic distributions) supported at a point, we introduce the hyperalgebra $\hy(G)$ of a locally $L$-analytic Lie group $G$ as a subalgebra of $D(G)$ (\Cref{Def - Hyperalgebra}).
We also consider the locally analytic adjoint representation of $G$ on $\hy(G)$ (\Cref{Prop - Adjoint representation}).
Furthermore, we show that if $G$ arises as the group of $L$-valued points of a linear algebraic group $\bG$, then $\hy\big(\bG(L) \big)$ coincides with the (algebraic) distribution algebra of $\bG$ as treated for example in \cite[Ch.\ I.7]{Jantzen03RepAlgGrp} (\Cref{Cor - Distributions of algebraic group}).

In \Cref{Sect - The functors FGP} we first study the subalgebra $D(\dot{\Fg},P) \subset D(G)$ and separately continuous modules over it.
In particular, we show that a separately continuous $D(\dot{\Fg},P)$-module structure is equivalent to separately continuous $\hy(G)$- and $D(P)$-module structures which satisfy compatibility conditions involving the adjoint representation of $P$ on $\hy(G)$ (\Cref{Prop - Equivalent characterization for composite algebra action}).
Then we define the functors $\dot{\CF}_P^G$ and prove the assertions of \Cref{Thm - A}.

Finally, in \Cref{Sect - Comparison} we first recall the definition of the category $\CO_\alg^\Fp$ in the $p$-adic setting and show how $M \in \CO_\alg^\Fp$ naturally becomes a separately continuous $D(\dot{\Fg},P)$-module (\Cref{Prop - Composite module structure on object in category O}). 
We then recapitulate the definition of the Orlik--Strauch functors $\CF_P^G$ and prove the comparison result of \Cref{Thm - B}.

\subsection{Acknowledgements}
This work originates from my doctoral thesis. 
I want to thank my advisor Sascha Orlik for his valuable advise and strong support.
Moreover, I want to thank Matthias Strauch and Yingying Wang for helpful comments and discussions.

This work was done while the author was a member of the Research Training Group \textit{Algebro-Geometric Methods in Algebra, Arithmetic and Topology} and subsequently of the Research Training Group \textit{Symmetries and Classifying Spaces – Analytic, Arithmetic and Derived} both funded by the DFG (Deutsche Forschungsgemeinschaft).

\subsection{Notation and Conventions}\label{Sect - Notation and conventions}

Let $L$ be a non-archimedean local field and $K$ an extension field of $L$ which we assume to be spherically complete. 
From \Cref{Sect - The functors FGP} on we demand the extension $K/L$ to be finite.

We let $\mathrm{LCS}_K$ denote the category of locally convex $K$-vector spaces; its morphisms are continuous homomorphisms of $K$-vector spaces.
We refer to \cite{Schneider02NonArchFunctAna} for an introduction to the theory of such locally convex vector spaces.
The category $\mathrm{LCS}_K$ is quasi-abelian in the sense of \cite{Schneiders98QuasiAbCat}, see \cite[Prop.\ 2.1.11]{Prosmans00DerCatFunctAna}.
The strict morphisms are those $f\colon V \ra W$ for which the induced $V/\Ker(f) \ra \Im(f)$ is a topological isomorphism.

For $V, W \in \mathrm{LCS}_K$ we let $\CL(V,W)$ denote the $K$-vector space of continuous homomorphisms from $V$ to $W$.
We write $\CL_b(V,W)$ and $\CL_s(V,W)$ for this space endowed with the strong respectively weak topology.
Moreover, we denote the dual space of a $K$-vector space $V$ by $V^\ast$.
If $V$ is locally convex, we write $V' \subset V^\ast$ for the subspace of continuous linear forms, as well as $V'_b$ and $V'_s$ for the strong and weak dual spaces accordingly.
However, if $V$ is a $K$-Banach space, we sometimes simplify the notation by writing $V'$ instead of $V'_b$.

By a locally $L$-analytic manifold we mean an analytic manifold in the sense of \cite[\S 5]{Bourbaki07VarDiffAnFasciDeResult} which in addition is Hausdorff, second-countable and finite-dimensional. 
Since $L$ is locally compact, it follows that such a locally $L$-analytic manifold is strictly paracompact, i.e.\ any open covering of it can be refined into a disjoint open covering \cite[Prop.\ 8.7]{Schneider11pAdicLieGrps}.

We fix the convention that all modules are assumed to be left modules and all algebras are assumed to be associative.

\section{Preliminaries}\label{Sect - Preliminaries}

\subsection{Preliminaries from Functional Analysis}

Recall that a locally convex $K$-vector space $V$ is \emph{pseudo-metriz\-able} (or a \emph{pseudo-metrizable space} here for short) if its topology can be defined by a pseudo-metric.
Analogously to \cite[Prop.\ 8.1]{Schneider02NonArchFunctAna} one shows that $V$ is pseudo-metrizable if and only if its topology can be defined by a countable family of seminorms.
One readily deduces the following properties.

\begin{lemma}\label{Lemma - Properties of pseudo-metrizable spaces}
	\begin{altenumerate}
		\item 
			Subspaces, quotients, countable products and the projective tensor product of pseudo-metrizable spaces are pseudo-metrizable.
			
		\item 
			A homomorphism from a pseudo-metrizable space to a locally convex $K$-vector space is continuous if and only if it is sequentially continuous.
			A pseudo-metrizable space is complete if and only if it is sequentially complete.
			
		\item 
			A pseudo-metrizable space is metrizable if and only if it is Hausdorff.
			
		\item 
			The Hausdorff completion of a pseudo-metrizable space is a $K$-Fr\'echet space.
	\end{altenumerate}
\end{lemma}

On the tensor product of $V, W \in \mathrm{LCS}_K$, we denote the projective (resp.\ inductive) tensor product topology by $V \ot{K,\pi} W$ (resp.\ $V \ot{K,\iota} W$).
We write $V \cot{K,\pi} W$ and $V \cot{K,\iota} W$ for the respective Hausdorff completions.
When the projective and inductive tensor product topology coincide, we will unambiguously write $V \ot{K} W$ and $V\cot{K} W$.
For instance, this is the case if both $V$ and $W$ are semi-complete LB-spaces (e.g.\ spaces of compact type) \cite[Prop.\ 1.1.31]{Emerton17LocAnVect} or under the following conditions.

\begin{lemma}\label{Lemma - Equality of projective and inductive tensor product}
	If $V, W \in \mathrm{LCS}_K$ are pseudo-metrizable and $V$ is barrelled, then
	\[ V \ot{K,\iota} W = V \ot{K,\pi} W .\]
\end{lemma}
\begin{proof}
	We slightly generalize the proof of \cite[Prop.\ 17.6]{Schneider02NonArchFunctAna} where this statement is shown when $V$ and $W$ are $K$-Fr\'echet spaces.
	As remarked in \textit{loc.\ cit.}, it suffices to show that any separately continuous bilinear map $\beta \colon V \times W \ra U$ into a locally convex $K$-vector space $U$ is continuous already.
	But $V \times W$ is pseudo-metrizable and therefore $\beta$ is continuous if and only if it is sequentially continuous.
	Moreover, the assumption that $V$ is barrelled ensures that the Banach--Steinhaus theorem \cite[Prop.\ 6.15]{Schneider02NonArchFunctAna} is applicable to $\CL_s(V,U)$.
	Now the claim follows in the same way as in the proof of \cite[Prop.\ 17.6]{Schneider02NonArchFunctAna}.
\end{proof}

We will also need the following technical lemma.

\begin{lemma}\label{Lemma - Annihilator of kernel is weak closure of image of the transpose}
	Let $f \colon V \ra W$ be a continuous homomorphism between locally convex $K$-vector spaces, and assume that $V$ is Hausdorff.
	Then in $V'$ we have the equality
	\[ \ov{\Im(\transp{f})}^s = \Ker(f)^\perp \defeq \big\{ \ell \in V' \,\big\vert\, \forall v \in \Ker(f) : \ell(v) = 0 \big\} \]
	where $\ov{\Im(\transp{f})}^s$ denotes the closure of the image of the transpose $\transp{f} \colon W' \ra V'$ in $V'_s$.
	Moreover, if in addition $V$ is semi-reflexive, then $\Ker(f)^\perp \subset \ov{\Im(\transp{f})}$ in $V'_b$.
\end{lemma}
\begin{proof}
	Taking the transpose twice, $f$ still defines a continuous homomorphism when $V$ and $W$ carry the respective weak topologies $V_s \defeq (V'_s)'_s$ and $W_s \defeq (W'_s)'_s$, see \cite[Thm.\ 7.3.3]{PerezGarciaSchikhof10LocConvSpNonArch}.
	Then the statement $\Ker(f)^\perp = \ov{\Im(\transp{f})}^s$ is part of \cite[p.II.47 Cor.\ 2]{Bourbaki87TopVectSp1to5}.

	It is a consequence of the Hahn--Banach theorem that the closed subspace $\ov{\Im(\transp{f})} \subset V'_b$ is closed for the weak topology on $V'_b$ as well \cite[Thm.\ 5.2.1]{PerezGarciaSchikhof10LocConvSpNonArch}.
	If we assume that $V$ is semi-reflexive, then this weak topology on the strong dual of $V$ coincides with the topology of the weak dual: $(V'_b)_s=V'_s $ \cite[Thm.\ 7.4.9]{PerezGarciaSchikhof10LocConvSpNonArch}.
	Since $\ov{\Im(\transp{f})}$ now is a closed subset of $V'_s$ which contains $\Im(\transp{f})$, it follows that $\ov{\Im(\transp{f})}^s \subset \ov{\Im(\transp{f})}$.
\end{proof}

\subsection{Locally Convex Algebras}

We now recall and define some notions concerning topological $K$-algebras.

\begin{definition}[{cf.\ \cite[Def.\ 1.2.1]{Emerton17LocAnVect}}]\label{Def - Topological algebras and modules}
	\begin{altenumerate}
		\item
			By a \emph{separately continuous locally convex $K$-algebra} we mean a locally convex $K$-vector space $A$ which carries the structure of a $K$-algebra such that the multiplication map $A\times A \ra A$ is separately continuous.
			If the multiplication map even is jointly continuous, we call $A$ a \emph{jointly continuous locally convex $K$-algebra}.
		\item
			Let $A$ be a separately continuous $K$-algebra.
			We define a \emph{separately continuous locally convex $A$-module} to be a locally convex $K$-vector space $M$ together with a compatible left $A$-module structure such that the scalar multiplication map $A\times M \ra M$ is separately continuous.
			If $A$ is a jointly continuous locally convex $K$-algebra and the scalar multiplication map is jointly continuous, we call such $M$ a \textit{jointly continuous locally convex $A$-module}.
	\end{altenumerate}
	In these contexts, we usually omit the adjective ``locally convex'' in the following.
\end{definition}

For a separately continuous $K$-algebra $A$, we let $\CM_A$	denote the category of separately continuous $A$-modules with morphisms given by continuous $A$-linear maps.
Moreover, we let $\CM_A^\mathrm{nF}$ denote the full additive subcategory of $\CM_A$ consisting of separately continuous $A$-modules whose underlying locally convex $K$-vector space is a nuclear $K$-Fr\'echet space.

\begin{proposition}\label{Prop - Category of nuclear Frechet modules is quasi-abelian}
	The category $\CM_A^\mathrm{nF}$ is quasi-abelian. 
	More precisely, for a morphism $f\colon M \ra N$ in $\CM_A^\mathrm{nF}$, its kernel (resp.\ cokernel) is given by $\Ker(f)$ (resp. $N/\ov{\Im(f)}$) with the induced subspace topology (resp.\ quotient topology).
	Moreover, the following are equivalent:
	\begin{altenumeratelevel2}
		\item 
		$f$ is strict in $\CM_A^\mathrm{nF}$,
		
		\item 
		$f$ considered as a morphism in $\mathrm{LCS}_K$ is strict,
		
		\item 
		$\Im(f)\subset N$ is a closed subspace.
	\end{altenumeratelevel2}
\end{proposition}
\begin{proof}
	First note that, for $N\in \CM_A$ and an $A$-submodule $N' \subset N$, continuity implies that $\ov{N'} \subset N$ is an $A$-submodule as well.
	Furthermore, restricting the separately continuous scalar multiplication on $N$ yields $N', N/N' \in \CM_A$.

	For a morphism $f\colon M \ra N$ in $\CM_A^\mathrm{nF}$, it follows from \cite[Prop.\ 8.3]{Schneider02NonArchFunctAna} and \cite[Prop.\ 19.4]{Schneider02NonArchFunctAna} that $\Ker(f)$ and $N/\ov{\Im(f)}$ are nuclear $K$-Fr\'echet spaces.
	Analogously to \cite[Prop.\ 4.4.2]{Prosmans00DerCatFunctAna} for the category of $K$-Fr\'echet spaces one verifies that $\Ker(f)$ and $N/\ov{\Im(f)}$ constitute the kernel and cokernel of $f$ in $\CM_A^\mathrm{nF}$ respectively.
	Furthermore, the forgetful functor from $\CM_A^\mathrm{nF}$ to the category of $K$-Fr\'echet spaces is continuous and cocontinuous.
	
	The equivalence of the characterizations of strictness of $f$ follows from the closed graph theorem \cite[Cor.\ 8.7]{Schneider02NonArchFunctAna}.
	Since the category of $K$-Fr\'echet spaces is quasi-abelian \cite[Prop.\ 3.4.5]{Prosmans00DerCatFunctAna}, we conclude that $\CM_A^\mathrm{nF}$ is quasi-abelian. 
\end{proof}

Let $A$ be a separately continuous $K$-algebra.
We can endow any finitely generated (abstract) $A$-module $M$ with the \emph{quotient topology} induced by some epimorphism $A^{\oplus n} \twoheadrightarrow M$ of $A$-modules.
This topology is locally convex.
It is independent of the choice of the $A$-module epimorphism as the following proposition shows when applied to $\id_M\colon M \ra M$.

\begin{proposition}\label{Prop - Topologies on finitely generated modules}
	Endowed with the quotient topology, $M$ becomes a separately continuous $A$-module.
	Furthermore, any homomorphism $f \colon M \ra M'$ of $A$-modules from a finitely generated $A$-module $M$ to $M'\in \CM_A$ is continuous when $M$ carries the quotient topology.
\end{proposition}
\begin{proof}
	For the first assertion it suffices to show that the multiplication map $A \times M \ra M$ is separately continuous.
	To this end, we consider the commutative diagram
	\begin{equation*}
		\begin{tikzcd}
			A \times A^{\oplus n} \ar[r] \ar[d, two heads] & A^{\oplus n} \ar[d, two heads] \\
			A \times M \ar[r] & M
		\end{tikzcd}
	\end{equation*}
	where the vertical maps are open by our choice of topology on $M$.
	But the multiplication map of the finite free $A$-module $A^{\oplus n}$ is separately continuous, which implies the claim for the multiplication map of $M$.

	Concerning a homomorphism $f$ as in the second assertion, we argue similarly to \cite[3.7.3 Prop.\ 2]{BoschGuentzerRemmert84NonArchAna}.
	Consider an epimorphism $\varphi \colon A^{\oplus n} \twoheadrightarrow M$ of $A$-modules which endows $M$ with its topology.
	Then the homomorphism $\varphi' \defeq f \circ \varphi$ is continuous because the addition and multiplication maps of $M' \in \CM_A$ are continuous and separately continuous respectively.
	As $\varphi$ is open by definition, it follows that $f$ is continuous.
\end{proof}

\begin{remark}
	When $A$ is a jointly continuous $K$-algebra, one analogously shows that every finitely generate $A$-module becomes a jointly continuous $A$-module with its quotient topology.
\end{remark}

It is natural to ask if a separately or jointly continuous $A$-module structure on $M$ extends to the Hausdorff completion $\widehat{M}$.
For jointly continuous $A$-modules this is true without further conditions \cite[Lemma 1.2.2]{Emerton17LocAnVect}.
Concerning separately continuous algebras and modules, we will need the following result.

\begin{lemma}\label{Lemma - Extending separately continuous module structure to completion}
	Let $A$ be a separately continuous $K$-algebra and $M \in \CM_A$.
	If $A$ is barrelled and $M$ is pseudo-metrizable, then the $A$-module structure on $M$ extends uniquely to a separately continuous $A$-module structure on $\widehat{M}$.
\end{lemma}
\begin{proof}
	By replacing $M$ with $M/\ov{\{0\}}$ we may assume that $M$ is Hausdorff and hence metrizable.
	Since $A$ is barrelled, it follows from \cite[p.III.31 Prop.\ 6]{Bourbaki87TopVectSp1to5} that the multiplication map $A \times M \ra M $ is $\mathfrak{S}$-hypocontinuous with respect to the set $\mathfrak{S}$ of all bounded subsets of $M$.
	Moreover, the assumption that $M$ is metrizable implies that every element of $\widehat{M}$ is the limit of a converging sequence in $M$.
	Such a sequence is bounded.
	Arguing like in the proof of \cite[p.III.32 Prop.\ 8]{Bourbaki87TopVectSp1to5} shows that the above map extends uniquely to the separately continuous multiplication map $A \times \widehat{M} \ra  \widehat{M} $.
	It follows from density of $M$ in $\widehat{M}$ that $\widehat{M}$ is an $A$-module.
\end{proof}

Now consider a continuous homomorphism $A \ra B$ of separately continuous $K$-algebras and $M \in \CM_A$.
There is an isomorphism of (abstract) $B$-modules
\begin{equation}\label{Eq - Tensor product as quotient}
	B \ot{A} M \cong \big( B \ot{K} M \big) \big/ M' ,
\end{equation}
where $M'$ is the $B$-submodule generated by $ba \otimes m - b \otimes a m$, for $b\in B$, $a\in A$, $m\in M$.

\begin{lemma}\label{Lemma - Base change for topological modules}
	\begin{altenumerate}
		\item
			The $B$-module $B \ot{A} M$ becomes a separately continuous $B$-module when endowed with the quotient topology induced from $B \ot{K,\iota} M$ via \eqref{Eq - Tensor product as quotient}.

		\item{\normalfont(\cite[Lemma 1.2.3]{Emerton17LocAnVect})}
			The analogous assertion holds in the setting with ``separately continuous'' replaced by ``jointly continuous'' and $B \ot{K,\iota} M$ by $B \ot{K,\pi} M$.
	\end{altenumerate}
\end{lemma}
\begin{proof}
	For (i), it suffices to show that the multiplication map of $B \ot{A} M$ is separately continuous.
	Continuity in $B$ is immediate.
	Since the quotient map $B\ot{K,\iota} M  \twoheadrightarrow B \ot{A} M$ is open, continuity in $B \ot{A} M$, for fixed $b'\in B$, follows once we show that $B\ot{K,\iota} M \ra B\ot{K,\iota} M $, $b\otimes m \mto b'b \otimes m$, is continuous.
	But this is the composite
	\begin{align*}
		B\ot{K,\iota} M &\xrightarrow{\mathmakebox[\widthof{\tiny $b'\otimes \blank$}]{b'\otimes \blank}} B\ot{K,\iota} \big( B\ot{K,\iota} M \big) \cong \big( B\ot{K,\iota} B \big) \ot{K,\iota} M \xrightarrow{m_B \otimes \id_M} B\ot{K,\iota} M, \\
		b \otimes m &\xmapsto{\mathmakebox[\widthof{\tiny $b'\otimes \blank$}]{}} b'\otimes b \otimes m ,
	\end{align*}
	where $m_B$ denotes the separately continuous homomorphism induced by multiplication in $B$.
	Moreover, it follows from the definition of the inductive tensor product topology \cite[\S 17.A]{Schneider02NonArchFunctAna} that the canonical isomorphism in the middle of the above composite is a topological isomorphism.
\end{proof}

\begin{definition}\label{Def - Base change for topological modules}
	In the above situation, we let $B \ot{A,\iota} M$ (resp.\ $B \ot{A,\pi} M$) denote the separately continuous (resp.\ jointly continuous) $B$-module $B \ot{A} M$ endowed with the respective quotient topology.
	Moreover, we write $B \cot{A,\iota} M$ (resp.\ $B \cot{A,\pi} M$) for its Hausdorff completion.	
\end{definition}

If $B$ and $M$ are pseudo-metrizable and at least one of them is barrelled, \Cref{Lemma - Equality of projective and inductive tensor product} implies that the projective and inductive tensor product topologies on $B \ot{K} M$ coincide.
Consequently we simply write $B \ot{A} M$ and $B \cot{A} M$ in such a case.

%Moreover, $B\ot{A}M$ is a locally convex $B$-module then.
%Indeed, in view of \Cref{Lemma - Equality of projective and inductive tensor product} it suffices to verify that $B\ot{A} M$ is pseudo-metrizable, which follows from \Cref{Lemma - Properties of pseudo-metrizable spaces} (i).

Recall that a locally convex $K$-vector space is \emph{hereditarily complete} if all its Hausdorff quotients are complete \cite[Def.\ 1.1.39]{Emerton17LocAnVect}.

\begin{lemma}\label{Lemma - Completion of projective tensor product over algebras}
	Let $A \ra B$ be a continuous homomorphisms of jointly continuous $K$-algebras and $M$ a jointly continuous $A$-module.
	If $B \cot{K,\pi} M$ is hereditarily complete, then 
	\begin{equation*}
		B \cot{A,\pi} M \cong  \big( B \cot{K,\pi} M \big) \big/ \ov{M'}
	\end{equation*}
	where $\ov{M'}$ denotes the closure in $ B \cot{K,\pi} M$ of the $B$-submodule $M'$ generated by $ba \otimes m - b \otimes a m$, for $b\in B$, $a\in A$, $m\in M$. 
\end{lemma}
\begin{proof}
	By \cite[Cor.\ 2.2]{BreuilHerzig18TowardsFinSlopePartGLn} completing preserves the strict exactness of 
	\begin{equation*}
		0 \lra M' \lra B \ot{K,\pi} M \lra B \ot{A,\pi} M \lra 0 .
	\end{equation*}
\end{proof}

The above condition on $B \cot{K,\pi} M$ is fulfilled for example when $B$ and $M$ both are $K$-Fr\'echet spaces (see discussion after \cite[Prop.\ 17.6]{Schneider02NonArchFunctAna}) or both are of compact type (see \cite[Prop.\ 1.1.32 (i)]{Emerton17LocAnVect}) by the comment after \cite[Def.\ 1.1.39]{Emerton17LocAnVect}.

\begin{lemma}\label{Lemma - Tensor identities for modules}
	\begin{altenumerate}
		\item
		For a separately continuous, unital $K$-algebra $A$ and $M \in \CM_A$, there is a topological isomorphism of separately continuous $A$-modules
		\begin{equation*}
			A \ot{A,\iota} M \overset{\cong}{\lra} M \,,\quad a \otimes m \lto am ,
		\end{equation*}
		
		\item
		Let $A$ and $B$ be separately continuous $K$-algebras, and let $L$, $M$ and $N$ be a separately continuous right $A$-module, $A$-$B$-bi-module and left $B$-module respectively.
		Then there is a canonical topological isomorphism
		\begin{equation*}
			\big( L \ot{A,\iota} M \big) \ot{B,\iota} N \cong L \ot{A,\iota} \big( M \ot{B,\iota} N \big) .
		\end{equation*}
	\end{altenumerate}
\end{lemma}
\begin{proof}
	In (i), it follows from the definition of the inductive tensor product topology that the homomorphism
	\begin{equation*}
		M \lra A \ot{K,\iota} M \longtwoheadrightarrow A \ot{A,\iota} M  \,,\quad m \lto 1 \otimes m ,
	\end{equation*}
	is continuous.
	Thus it constitutes an inverse to the claimed topological isomorphism.

	For (ii), we have a topological isomorphism of locally convex $K$-vector spaces
	\begin{equation*}
		\varphi \colon 	\big( L \ot{K,\iota} M \big) \ot{K,\iota} N \overset{\cong}{\lra} L \ot{K,\iota} \big( M \ot{K,\iota} N \big) 
	\end{equation*}
	like in the proof of \Cref{Lemma - Base change for topological modules} (i).
	We pass to the quotients
	\begin{align*}
		\pi_l \colon &\big(L \ot{K,\iota} M \big) \ot{K,\iota} N \lra \big(L \ot{A,\iota} M \big)\ot{K,\iota} N  \lra \big(L \ot{A,\iota}  M \big)\ot{B,\iota} N  , \\
		\pi_r \colon &L \ot{K,\iota} \big(M \ot{K,\iota} N \big) \lra L \ot{K,\iota} \big(M \ot{B,\iota} N \big) \lra L \ot{A,\iota} \big( M \ot{B,\iota} N \big) .
	\end{align*}
	Then $\pi_r \circ \varphi$ factors over $\pi_l$, and $\pi_l \circ \varphi^{-1} $ factors over $\pi_r$ yielding the sought topological isomorphism.
\end{proof}

\section{Locally Analytic Representation Theory}\label{Sect - Locally analytic representation theory}

In this section we recapitulate some aspects of the theory of locally analytic representations. 
Such representations were systematically introduced and studied by F\'eaux de Lacroix \cite{FeauxdeLacroix99TopDarstpAdischLieGrp}, and Schneider and Teitelbaum \cite{SchneiderTeitelbaumPrasad01UgFinLocAnRep, SchneiderTeitelbaum02LocAnDistApplToGL2, SchneiderTeitelbaum03AlgpAdicDistAdmRep, SchneiderTeitelbaum05DualAdmLocAnRep} originally over $p$-adic fields.
However, many foundational definitions and results readily generalize to non-archimedean local fields of arbitrary characteristic, which was already remarked and employed by Gräf \cite[App.\ A]{Graef20BoundDistr}.
For propositions first stated in the $p$-adic setting, we will not include proofs if their original proofs remain valid in positive characteristic.

\subsection{Locally Analytic Functions}

Let $X$ be a locally $L$-analytic manifold (see the convention in \Cref{Sect - Notation and conventions}) and let $V$ be a Hausdorff locally convex $K$-vector space.

Recall that a \emph{BH-subspace} of $V$ is an algebraic subspace $E\subset V$ which admits a compatible structure of a $K$-Banach space such that the associated topology is finer than its subspace topology \cite[Def.\ 1.2.1]{FeauxdeLacroix99TopDarstpAdischLieGrp}.
We let $\ov{E}$ denote $E$ carrying such a Banach topology, which is necessarily unique.

\begin{definition}[{cf.\ \cite[\S2]{SchneiderTeitelbaum02LocAnDistApplToGL2}}]\label{Def - Locally analytic functions}
	\begin{altenumerate}
		\item 
			A \emph{$V$-index $\CI$ on $X$} is a family of triples $(U_i,\varphi_i,E_i)_{i\in I}$ where 
			\begin{altenumeratelevel2}
				\item 
					$X= \bigcup_{i\in I} U_i$ is a disjoint open covering,
					
				\item 
					$\varphi_i\colon U_i \ra  L^{d_i}$ is a chart of $X$ whose image is an affinoid polydisc,
					
				\item 
					 and $E_i \subset V$ is a BH-subspace.
			\end{altenumeratelevel2} 
			The $V$-indices on $X$ form a directed set \cite[Bem.\ 2.1.9]{FeauxdeLacroix99TopDarstpAdischLieGrp}.
			Moreover, $C^\an (U_i,E_i)$ is defined to be the space of functions $f\colon U_i \ra E_i$ such that $f \circ \varphi_i^{-1} $ is given by an $\ov{E_i}$-valued power series strictly convergent on $\varphi_i(U_i)$, and topologized via
			\begin{equation*}
				C^\an (U_i,E_i)\cong \CO(\varphi_i(U_i)) \cot{K,\pi} \ov{E_i} .
			\end{equation*}
			Here $\CO(\varphi_i(U_i))$ denotes the $K$-Banach algebra of power series with coefficients in $K$ which converge on all points of $\varphi_i(U_i)$ defined over an algebraic closure of $K$.

		\item 
			The locally convex $K$-vector space of \emph{locally analytic functions on $X$ with values in $V$} is defined to be the locally convex inductive limit
			\begin{equation*}
				C^\la(X,V) \defeq \varinjlim_{\CI} \, C^\la_\CI (X,V) 
			\end{equation*}
			over all $V$-indices on $X$ where $C^\la_\CI (X,V) \defeq \prod_{i\in I} C^\an (U_i,E_i)$.
	\end{altenumerate}
\end{definition}

An element $f\in C^\la(X,V)$ defines a function from $X$ to $V$, which by the identity theorem for power series determines $f$ uniquely.
Thus we may identify $f$ with its induced function $f\colon X \ra V$.
We call functions arising this way \emph{locally analytic}.
They are continuous in particular.

\begin{proposition}\label{Prop - Properties of locally analytic functions}
	\begin{altenumerate}
		\item
			The formation of $C^\la(X,V)$ is contravariantly functorial in $X$ and covariantly functorial in $V$ {\normalfont \cite[Bem.\ 2.1.11]{FeauxdeLacroix99TopDarstpAdischLieGrp}, \cite[Prop.\ 1.1.7]{Emerton17LocAnVect}}.
			
		\item 
			The evaluation homomorphisms $\ev_x \colon C^\la(X,V) \ra V$, $f \mto f(x)$, for $x\in X$, are continuous, and $C^\la(X,V)$ is Hausdorff and barrelled {\normalfont \cite[Satz 2.1.10]{FeauxdeLacroix99TopDarstpAdischLieGrp}}.
			
		\item 
			Any disjoint covering $X= \bigcup_{i\in I} X_i$ by open subsets yields a topological isomorphism {\normalfont \cite[Kor.\ 2.2.4]{FeauxdeLacroix99TopDarstpAdischLieGrp}}
			\begin{equation*}
				C^\la(X,V) \overset{\cong}{\lra} \prod_{i\in I} C^\la(X_i,V) \,,\quad f \lto (f\res{X_i})_{i\in I} .
			\end{equation*}
					
		\item 
			If $V$ is of compact type, $C^\la(X,V)$ is reflexive and complete, cf.\ {\normalfont \cite[p.\ 40]{Emerton17LocAnVect}}.
			If in addition $X$ is compact, then $C^\la(X,V)$ is of compact type with $C^\la(X,V) \cong C^\la(X,K) \cot{K} V$ {\normalfont \cite[Satz 2.3.2]{FeauxdeLacroix99TopDarstpAdischLieGrp}}.
		
		\item 
			For compact locally $L$-analytic manifolds $X$ and $Y$, there are topological isomorphisms {\normalfont \cite[Lemma A.1]{SchneiderTeitelbaum05DualAdmLocAnRep}}
			\begin{alignat*}{3}
				C^\la(X,K) \cot{K} C^\la(Y,K) &\lra C^\la\big(X,C^\la(Y ,K&&) \big) &&\lra C^\la(X\times Y ,K) , \\
				f \otimes g &\lto f(\blank)g  \,, && F	 &&\lto \big[ (x,y) \mto F(x)(y) \big] .
			\end{alignat*}
			
	\end{altenumerate}
\end{proposition}

\begin{definition}
	The space of \emph{locally analytic distributions on $X$} is defined as the strong dual space
	\begin{equation*}
		D(X,K) \defeq C^\la(X,K)'_b .
	\end{equation*}
	When the coefficient field $K$ is apparent from the context, we simply write $D(X)$.
	The element $\delta_x \defeq \ev_x \in D(X)$, for $x\in X$, is called \emph{Dirac distribution supported in $x$}.
\end{definition}

\begin{proposition}\label{Prop - Properties of space of distributions}
	\begin{altenumerate}
		\item 
			The locally convex $K$-vector space $D(X)$ is reflexive and barrelled.
			If $X$ is compact, $D(X)$ is a nuclear $K$-Fr\'echet space {\normalfont \cite[Thm.\ 1.1, Thm.\ 1.3]{SchneiderTeitelbaum02LocAnDistApplToGL2}}.
			
		\item 
			Any disjoint covering $X= \bigcup_{i\in I} X_i$ by open subsets yields a topological isomorphism {\normalfont \cite[p.\ 447]{SchneiderTeitelbaum02LocAnDistApplToGL2}}
			\begin{equation*}
				D(X) \cong \bigoplus_{i\in I} D(X_i) .
			\end{equation*}
			
		\item 
			For locally $L$-analytic manifolds $X$ and $Y$, there is a canonical topological isomorphism {\normalfont \cite[Prop.\ A.3]{SchneiderTeitelbaum05DualAdmLocAnRep}}
			\begin{equation*}
				D(X\times Y) \cong D(X) \cot{K,\iota} D(Y) .
			\end{equation*}
			
		\item 
			The subspace of $D(X)$ generated by all Dirac distributions $\delta_x$, for $x\in X$, is dense {\normalfont \cite[Lemma 1.1.1]{Kohlhaase05InvDistpAdicAnGrp}}.
			
	\end{altenumerate}
\end{proposition}

In particular, every $\delta \in D(X)$ is supported in some compact open subset $C$ of $X$, i.e.\ contained in $D(C) \hookrightarrow D(X)$. 
Indeed, there exists a disjoint covering of $X$ by compact open subsets, to which we may apply (ii).

\begin{lemma}\label{Lemma - Dual of tensor product with locally analytic functions}
	If $V$ is of compact type, there is a topological isomorphism
	\begin{equation*}
		D(X) \cot{K,\iota} V'_b \overset{\cong}{\lra} \big( C^\la(X,K) \cot{K,\pi} V \big)'_b \,,\quad \delta \otimes \ell \lto [f \otimes v \mto \delta(f) \ell(v)] .
	\end{equation*}
\end{lemma}
\begin{proof}
	When $X$ is compact, the statement follows from \cite[Prop.\ 1.1.32 (ii)]{Emerton17LocAnVect}.
	In the general case one considers a disjoint covering by compact open subsets.
	The claim then follows from the compact case by distributivity of $\cot{K,\iota}$ over direct sums \cite[Cor.\ 1.2.14]{Kohlhaase05InvDistpAdicAnGrp}, resp.\ of $\cot{K,\pi}$ over direct products \cite[Lemma 2.1 (iii)]{BreuilHerzig18TowardsFinSlopePartGLn}, and the duality between direct sums and products \cite[Prop.\ 9.11]{Schneider02NonArchFunctAna}.
\end{proof}

\begin{theorem}[{\cite[Thm.\ 2.2]{SchneiderTeitelbaum02LocAnDistApplToGL2}}]\label{Thm - Integration map}
	There exists a continuous $K$-linear \emph{integration map}
	\begin{equation*}
		I \colon C^\la (X,V) \lra \CL_b \big(D(X),V \big)
	\end{equation*}
	uniquely defined by $I(f)(\delta_x) = f(x)$, for all $f\in C^\la (X,V)$, $x \in X$.
	Moreover, this map is natural in $X$ and $V$, and injective.
	If $V$ is of LB-type\footnote{Recall that $V$ is \emph{of LB-type} if it can be written as the increasing union of countably many BH-subspaces \cite[Def.\ 1.1.9]{Emerton17LocAnVect}.}, then $I$ is an isomorphism of (abstract) $K$-vector spaces with inverse 
	\begin{equation*}
		I^{-1} \colon \CL \big(D(X),V \big) \lra C^\la (X,V) \,,\quad T \lto \big[x \mto T(\delta_x) \big] .
	\end{equation*}
\end{theorem}

\begin{corollary}\label{Cor - Pairing of distributions and vector valued functions}
	There is a natural, separately continuous, non-degenerate $K$-bilinear pairing
	\begin{equation*}
		D(X) \times C^\la(X,V) \lra V \,,\quad (\delta, f) \lto \delta(f) \defeq I(f)(\delta).
	\end{equation*}
	This pairing is compatible with the duality between $D(X)$ and $C^\la(X,K)$ in the sense that, for compact open $U\subset X$ and a BH-subspace $E \subset V$, its restriction to $D(U) \times C^\la(U,E)$ is given by tensoring the duality pairing $D(U) \times C^\la(U,K)\ra K$ with $E$.
\end{corollary}
\begin{proof}
	It is clear that the above pairing defined by $I$ is natural, $K$-bilinear, and separately continuous.
	On a compact subset $U\subset X$ and for a $K$-Banach space $E$, the map $I$ is defined as the continuous linear bijection 
	\begin{equation*}
		\begin{tikzcd}[row sep = 0ex	,/tikz/column 1/.append style={anchor=base east}	,/tikz/column 3/.append style={anchor=base west}]
			I \colon C^\la(X,E) \ar[r] & \eqmathbox[A]{C^\la(X,K) \cot{K} E} \ar[r, "\cong"] & \CL_b \big(D(X),E \big) ,  \\
			f(\blank) \, v  \ar[r, mapsfrom]& \eqmathbox[A]{f\otimes v} \ar[r, mapsto]& \big[\delta \mto \delta(f) \,v\big] ,
		\end{tikzcd}
	\end{equation*}
	see the proof of \cite[Thm.\ 2.2]{SchneiderTeitelbaum02LocAnDistApplToGL2}.
	This shows the last statement.

	To show the non-degeneracy of the pairing, we first fix a distribution $\delta \in D(X)$ and assume that $\delta(f)=0$, for all $f\in C^\la(X,V)$.
	Let $\delta$ be supported on a compact open subset $U\subset X$.
	If $V \neq \{0\}$, then $\delta = 0$ follows from the non-degeneracy of the duality pairing $D(U) \times C^\la(U,K)\ra K$ and the compatiblity already shown.
	On the other hand, consider $f \in C^\la(X,V)$ such that $\delta(f)=0$, for all $\delta \in D(X)$.
	Then we have $f(x)= I(f)(\delta_x) =0$, for all Dirac distributions $\delta_x$, and hence $f=0$.
\end{proof}

\subsection{Locally Analytic Representations}

Let $G$ be a \emph{locally $L$-analytic Lie group}, i.e.\ locally $L$-analytic manifold which carries the structure of a group such that the multiplication and inversion maps
\begin{equation*}
	m \colon G \times G \lra G \,,\qquad \inv\colon G \lra G
\end{equation*}
are locally $L$-analytic.
As a topological group $G$ is Hausdorff, totally disconnected and locally compact. 
This implies that each neighbourhood of the identity element $1$ in $ G$ contains a compact open subgroup.

\begin{definition}[{\cite[Def.\ p.451]{SchneiderTeitelbaum02LocAnDistApplToGL2}}]\label{Def - Locally analytic representation}
	A \textit{(left) locally analytic $G$-representation} is a barrelled, Hausdorff locally convex $K$-vector space $V$ with a $G$-action by continuous endomorphisms such that the orbit maps $G \ra V$, $g \mto g.v$, are locally analytic in the sense of \Cref{Def - Locally analytic functions}, for all $v\in V$.

	We let $\Rep_K^{\la}(G)$ denote the category of locally analytic $G$-representations with morphism given by $G$-equivariant continuous homomorphism of $K$-vector spaces.
	We write $\Rep_K^{\la,\mathrm{ct}}(G)$ for the full subcategory of locally analytic $G$-representations whose underlying locally convex $K$-vector space is of compact type.
\end{definition}

\begin{example}[{cf.\ \cite[Bsp.\ 3.1.6]{FeauxdeLacroix99TopDarstpAdischLieGrp}}]\label{Expl - Regular representations}
		Let $G$ be compact and let $V \in \mathrm{LCS}_K$ be Hausdorff.
		Then
		\begin{equation*}
			G \times C^\la(G,V) \lra C^\la(G,V) \,,\quad (g,f) \lto f(g^{-1}  \blank) \defeq \big[h \mto f(g^{-1}h) \big],
		\end{equation*}
		defines the \textit{left regular $G$-representation}, which is locally analytic.
		Similarly, the \textit{right regular $G$-representation} on $C^\la(G,V)$ and the \textit{$G$-representation by conjugation}
		\begin{equation*}
			G \times C^\la(G,V) \lra C^\la(G,V) \,,\quad (g,f) \lto f(g^{-1} \blank g),
		\end{equation*}
		are locally analytic $G$-representations.
\end{example}

\begin{example}[{cf.\ \cite[Sect.\ 2]{SchneiderTeitelbaumPrasad01UgFinLocAnRep}}]
	Let $V$ be an admissible smooth $G$-representation on a $K$-vector space. 
	The unit element of $G$ has a countable neighbourhood basis consisting of compact open subgroups $(H_n)_{n\in \BN}$.
	Since $V = \bigcup_{n\in \BN} V^{H_n}$ and the $V^{H_n}$ are finite-dimensional by assumption, it follows that $V$ endowed with the finest locally convex topology is of compact type. 
	We obtain $V \in \Rep_K^{\la,\mathrm{ct}}(G)$ with this topology because the orbit maps $g \mto g.v$, for $v\in V$, are locally constant.
\end{example}

\begin{proposition}[{\cite[Prop.\ 3.6.14]{Emerton17LocAnVect}}]\label{Prop - Subrepresentations and quotients of locally analytic representations}
	Let $V$ be a locally analytic $G$-representation and $W\subset V$ a $G$-invariant closed subspace.
	\begin{altenumerate}
		\item
			Then $V/W$ is a locally analytic $G$-representation with respect to the induced $G$-action.
		\item
			If $W$ is barrelled, then $W$ is a locally analytic $G$-representation with respect to the induced $G$-action.	
	\end{altenumerate}
\end{proposition}

\begin{lemma}[{cf.\ discussion after \cite[Thm.\ 3.6.12]{Emerton17LocAnVect}}]\label{Lemma - Continuity of the orbit homomorphism}
	Let $V \in \Rep_K^{\la}(G)$ with orbit maps $\rho_v$, $v\in V$.
	If $V$ is an LB-space\footnote{Recall that $V$ is an \emph{LB-space} if it is topologically isomorphic to the locally convex inductive limit of a sequence of $K$-Banach spaces \cite[Def.\ 1.1.16 (i)]{Emerton17LocAnVect}.} then the orbit homomorphism
	\begin{equation*}
		o_V\colon V \lra C^\la(G,V) \,,\quad v \lto \rho_v , 
	\end{equation*}
	is continuous.
\end{lemma}
\begin{proof}
	Let $H\subset G$ be a compact open subgroup and $g_i \in G$ such that $G = \bigcup_{i\in I} H g_i$ is a disjoint covering.
	Under the topological isomorphism from \Cref{Prop - Properties of locally analytic functions} (iii), the homomorphism $o_V$ coincides with the product of the maps $o_i \colon V \ra C^\la(H g_i, V)$, $v \mto \rho_v\res{H g_i}$.
	Hence it suffices to show that all $o_i$ are continuous.

	We fix $i\in I$ and consider the graph $\Gamma_{o'} \subset V \times C^\la(H',V)$ of $o' \defeq o_i$, for $H' \defeq H g_i$.
	But $\Gamma_{o'}$ is precisely the kernel of the continuous homomorphism
	\begin{equation*}
		V \times C^\la(H',V) \lra \prod_{h\in H} V \,,\quad (v,f) \lto \big(\rho_v(hg_i) - f(hg_i)\big)_{h\in H} .
	\end{equation*}
	Therefore $\Gamma_{o'} \subset V \times C^\la(H',V)$ is closed, and we conclude by a version of the closed graph theorem \cite[p.II.34 Prop.\ 10]{Bourbaki87TopVectSp1to5} that $o'$ is continuous.
	Concerning this last step we remark that in particular $V$ is of LB-type so that $C^\la(H',V)$ is an LB-space by \cite[Prop.\ 2.1.30]{Emerton17LocAnVect}.
\end{proof}

\begin{proposition}\label{Prop - Locally analytic action on Hom space}
	Let $V$ be a locally analytic $G$-representation on a $K$-Banach space and let $W$ be a $K$-Banach space. 
	Then the $G$-representation
	\begin{equation*}
		\eta \colon G \times \CL_b(V,W) \lra \CL_b(V,W) \,,\quad (g,\varphi) \lto \big[v \mto \varphi(g^{-1}.v) \big],
	\end{equation*}
	is locally analytic.
\end{proposition}
\begin{proof}
	Each fixed $g\in G$ acts on $\CL_b(V,W)$ by composition with a continuous automorphism of $V$.
	Hence this is an action via continuous automorphisms \cite[p.113]{Schneider02NonArchFunctAna}.
	It remains to show that the orbit maps of $\eta$ are locally analytic.
	Consider
	\begin{equation*}
		\CL (V,W) \lra \CL \big(C^\la(G,V),C^\la(G,W) \big) \overset{o_V^\ast}{\lra} \CL \big(V,C^\la(G,W) \big) \xrightarrow{(I_W)_\ast} \CL \big(V, \CL_b(D(G),W) \big)
	\end{equation*}
	where the first map is induced by functoriality from \Cref{Prop - Properties of locally analytic functions} (i), the second by the orbit homomorphism $o_V \colon V \ra C^\la(G,V)$ from \Cref{Lemma - Continuity of the orbit homomorphism}, and the third by the integration homomorphism $I_W \colon C^\la (G,W) \ra \CL_b(D(G),W)$ from \Cref{Thm - Integration map}.
	Moreover, we have
	\begin{equation*}
		\CL\big(V, \CL_b(D(G),W) \big) \lra \CL \big(D(G), \CL_b(V,W) \big) \xrightarrow{I^{-1}_{\CL_b(V,W)}} C^\la(G,\CL_b(V,W))
	\end{equation*}
	where the first map is via \cite[Prop.\ 1.1.35]{Emerton17LocAnVect}.
	One verifies that the composition of all of these maps sends $\varphi \in \CL(V,W)$ to $\eta_\varphi \circ \inv$.
	Therefore  the orbit map $\eta_\varphi$ is locally analytic.
\end{proof}

\begin{corollary}[{\cite[Prop.\ 5.1.11]{Emerton17LocAnVect}}]\label{Cor - Contragredient representation on Banach space is locally analytic}
	Let $V$ be a locally analytic $G$-representation on a $K$-Banach space.
	Then the contragredient $G$-representation on $V'_b$ is locally analytic.
\end{corollary}

\begin{corollary}[{For $\mathrm{char}(L)=0$, cf.\ \cite[Kor.\ 3.1.9]{FeauxdeLacroix99TopDarstpAdischLieGrp}}]\label{Cor - Locally analytic Banach space representation is given by homomorphism of Lie groups}
	A $G$-representation on a $K$-Banach space $V$ is locally analytic if and only if it is given by a locally $L$-analytic homomorphism $\rho\colon G \ra \GL(V)$ of Lie groups\footnote{The latter is called an analytic linear representation by Bourbaki \cite[III.\ \S 1.2 Expl.\ (3)]{Bourbaki89LieGrpLieAlg1to3}. Concerning the definition of locally analytic manifolds whose charts take values in Banach spaces, see \cite[\S 5.1]{Bourbaki07VarDiffAnFasciDeResult}.} where we consider $\GL(V) \subset \CL_b(V,V)$ as a locally $L$-analytic Lie group via restriction of scalars.
\end{corollary}
\begin{proof}
	If the representation $V$ is a locally analytic, we may apply \Cref{Prop - Locally analytic action on Hom space} with $W\defeq V$.
	Then the orbit map $\eta_{\id_V}$ is locally analytic and the $G$-action on $V$ is given by $\rho = \eta_{\id_V} \circ \inv\colon G \ra \GL(V)$.
	For the converse implication, note that the orbit map $\rho_v$, for $v\in V$, is the composition of $\rho$ with the evaluation homomorphism $\ev_v \colon \GL(V) \ra V$, which is continuous.	
\end{proof}

\subsection{Locally Analytic Distribution Algebras}

Locally analytic representations of $G$ can be understood in terms of modules over the locally analytic distribution algebra $D(G)$.
Here the algebra structure on $D(G)$ is given by the convolution product
\begin{equation*}
	\ast \colon D(G) \times D(G) \lra D(G)\cot{K,\iota} D(G) \cong D(G\times G) \xrightarrow{m_\ast} D(G) 
\end{equation*}
induced by the multiplication map of $G$ and \Cref{Prop - Properties of space of distributions} (iii).

\begin{proposition}[{\cite[Prop.\ 2.3]{SchneiderTeitelbaum02LocAnDistApplToGL2}, \cite[Rmk.\ A.4]{SchneiderTeitelbaum05DualAdmLocAnRep}}]\label{Prop - Convolution product}
	The convolution product $\ast$ endows $D(G)$ with the structure of a unital, separately continuous $K$-algebra.
	Its unit element is $\delta_1$ where $1$ is the identity element of $G$, and we have $\delta_g \ast \delta_{g'} = \delta_{gg'}$, for $g,g'\in G$.
	If $G$ is compact then $\ast$ is jointly continuous.
\end{proposition}

\begin{lemma}[{cf.\ \cite[Proof of Prop.\ 3.5]{OrlikStrauch15JordanHoelderSerLocAnRep}}]\label{Lemma - Fubinis theorem}
	For $f \in C^\la(G,K)$ and $\mu , \nu \in D(G)$, the functions on $G$ defined by
	\begin{align*}
		g \lto \nu \big[ g' \mto f(gg') \big] \qquad\text{and}\qquad g' \lto \mu \big[ g \mto f(gg') \big]
	\end{align*}
	are locally analytic, and we have the following identities reminiscent of Fubini's theorem:
	\begin{equation*}
		\begin{aligned}
			(\mu \ast \nu )(f) = \mu \big[g \mto \nu\big[g'\mto f(gg') \big]\big]
			= \nu\big[ g'\mto \mu\big[g \mto f(gg') \big]\big].
		\end{aligned}
	\end{equation*}
\end{lemma}
\begin{proof}
	We may assume that $\nu$ is supported in the compact open subset $H'\subset G$.
	Then, restricted to some compact open subset $H\subset G$, the first function is the image of $f\res{m(H\times H')}$ under
	\begin{equation*}
			C^\la \big(m(H\times H'),K \big)\xrightarrow{m^\ast} C^\la(H \times H', K) \cong C^\la \big(H, C^\la(H',K) \big) \xrightarrow{\nu_\ast} C^\la(H,K) 
	\end{equation*}
	using \Cref{Prop - Properties of locally analytic functions} (v).
	Analogously, one shows that the second function is locally analytic.

	Now let $\mu$ be supported in the compact open subset $H \subset G$.
	By the definition of $\ast$ the distribution $\mu \ast \nu$ is supported in $m(H \times H')$ and given by
	\begin{equation}\label{Eq - Explicit map for convolution of distributions}
		C^\la \big(m(H \times H'),K \big) \xrightarrow{m^\ast} C^\la(H\times H',K) \cong C^\la(H,K) \cot{K} C^\la (H',K) \xrightarrow{\mu \otimes \nu} K . 
	\end{equation}
	The claimed identities then follow from the commutativity of
	\begin{equation*}
		\begin{tikzcd}[row sep = 0ex]
			&C^\la(H,K) \cot{K} C^\la(H',K) \ar[rd, start anchor=east, end anchor=145, "\mu \otimes \nu"] & \\
			C^\la(H \times H', K) \ar[ru, end anchor=west, "\cong"] \ar[rd, start anchor=south east, end anchor = 178, "\cong"] &&K \\
			&C^\la \big(H, C^\la(H',K) \big) \overset{\nu_\ast}{\lra}  C^\la(H,K) \ar[ru, start anchor=east, "\mu"'] &
		\end{tikzcd}
	\end{equation*}
	and the analogous diagram for $\nu \circ \mu_\ast$.
\end{proof}

\begin{definition}\label{Def - Involution of distriubtion algebra}
	The locally $L$-analytic anti-automorphism $\inv \colon G \ra G$, $ g \mto g^{-1},$ induces by functoriality a topological automorphism
	\begin{equation*}
		D(G) \lra D(G) \,,\quad \delta \lto \dot{\delta} \defeq \delta \circ \inv^\ast = \big[f \mto \delta(f \circ \inv)\big] .
	\end{equation*}
	For $\gamma, \delta \in D(G)$, one computes that $\left( \gamma \ast \delta \right)\dot{} = \dot{\delta} \ast \dot{\gamma}$.
\end{definition}

\begin{proposition}[{\cite[Prop.\ 3.2, Cor.\ 3.3, Cor.\ 3.4]{SchneiderTeitelbaum02LocAnDistApplToGL2}}]\label{Prop - Anti-equivalence for locally analytic representations}
	Let $V$ be a locally analytic $G$-re\-pre\-sen\-tation with orbit maps $\rho_v \colon G \ra V$, for $v\in V$.
	\begin{altenumerate}
		\item
			Then $V$ becomes a separately continuous $D(G)$-module via $\delta \ast v \defeq I(\rho_v)(\delta)$, for $\delta \in D(G)$, $v \in V$.
			In particular, we have $\delta_g \ast v = g.v$.
			This yields an equivalence of categories
			\begin{equation*}
				\Bigg(\substack{\text{\small locally analytic $G$-representations}  \\ \text{\small of LB-type with continuous} \\ \text{\small  $G$-equivariant homomorphisms}}\Bigg) \overset{\cong}{\lra}
				\left(\substack{\text{\small separately continuous $D(G)$-modules}  \\ \text{\small of LB-type with continuous} \\ \text{\small $D(G)$-module maps}}\right).
			\end{equation*}
			
		\item
			Moreover, $V'_b$ becomes a separately continuous $D(G)$-module via $\delta \ast \ell \defeq [v\mto \ell(\dot{\delta} \ast v) ]$, for $\delta \in D(G)$, $\ell \in V'_b$.
			In particular, we have $\delta_g \ast \ell = g.\ell$ with respect to the contragredient $G$-action on $V'_b$.
			This induces an anti-equivalence of categories
			\begin{align*}
				\Rep_K^{\la,\,\mathrm{ct}}(G) \overset{\cong}{\lra} \CM_{D(G)}^\mathrm{nF} \,,\quad \left[f\colon V \ra W \right] \lto \left[\transp{f} \colon W'_b \ra V'_b \right] . 
			\end{align*}
			If $G$ is compact, then the multiplication for $D(G)$-modules belonging to the right hand side category is jointly continuous automatically.
	\end{altenumerate}
\end{proposition}

\begin{remark}
	It follows from this anti-equivalence of categories that $\Rep_K^{\la,\mathrm{ct}}(G)$ is quasi-abelian.
	The kernel (resp.\ cokernel) of a morphism $f \colon M \ra N$ in $\Rep_K^{\la,\mathrm{ct}}(G)$ is given by $\Ker(f)$ (resp.\ $N/\ov{\Im(f)}$) with the subspace (resp.\ quotient) topology. 
	Indeed, there are canonical identifications $ \Ker(f) \cong \big( M'_b / \ov{\Im(\transp{f})} \big)'_b$ and $N/\ov{\Im(f)} \cong \Ker(\transp{f})'_b$, see p.II.47 Corollary 2, p.IV.8 Proposition 9 (ii) and p.IV.10 Remark in \cite{Bourbaki87TopVectSp1to5}.
\end{remark}

\subsection{Locally Analytic Induction}

Here let $G$ be a locally $L$-analytic Lie group and let $H\subset G$ be a \emph{locally $L$-analytic Lie subgroup}, i.e.\ a subgroup which is a locally $L$-analytic submanifold.
The latter means that, for every $h\in H$, there exist a chart $\varphi \colon U \ra L^d$ around $h$ and a linear subspace $F \subset L^d$ such that $\varphi$ induces a homeomorphism
\begin{equation*}
	\varphi \res{U \cap H} \colon U \cap H \lra \varphi(U) \cap F ,
\end{equation*}
see \cite[5.8.1,3]{Bourbaki07VarDiffAnFasciDeResult}.
Such a subgroup acquires the structure of a locally $L$-analytic Lie group itself and is closed in $G$ necessarily \cite[5.12.3]{Bourbaki07VarDiffAnFasciDeResult}.

\begin{definition}[{\cite[Kap.\ 4.1]{FeauxdeLacroix99TopDarstpAdischLieGrp}}]\label{Def - Definition of locally analytic induction}
	For a locally analytic $H$-representation $V$ we define the \emph{locally analytically induced representation}
	\begin{equation*}
		\Ind^{\la,G}_H (V) \defeq \left\{ f \in C^\la(G,V) \middle{|} \forall g\in G, h \in H: f(gh)= h^{-1}.f(g) \right\} \subset C^\la(G,V)
	\end{equation*}
	and consider it with the left regular $G$-action.
\end{definition}

\begin{proposition}\label{Prop - Existence of locally analytic induction}
	Let $V$ be a locally analytic $H$-representation of compact type and assume that there exists a compact open subgroup $G_0 \subset G$ such that $G = G_0 H$.
	Then $\Ind^{\la,G}_H(V)$ is a locally analytic $G$-representation of compact type.
\end{proposition}
\begin{proof}
	It is the content of \cite[Satz 4.1.5]{FeauxdeLacroix99TopDarstpAdischLieGrp} that $G$ acts on $\Ind^{\la,G}_H(V)$ by continuous automorphisms and with locally analytic orbit maps.
	To see that $\Ind^{\la,G}_H(V)$ is of compact type one argues analogously to \cite[Sect.\ 2.1]{Emerton07JacquetModLocAnRep2}:
	We fix $G_0$ as above and set $H_0 \defeq G_0 \cap H$.
	We may view $V$ as a locally analytic $H_0$-representation and have a continuous homomorphism
	\begin{equation*}
		\Ind_H^{\la, G} (V) \lra \Ind_{H_0}^{\la,G_0} (V) \,,\quad f \lto f\res{G_0}.
	\end{equation*}
	Let $h_i$, for $i\in I$, be representatives of $H_0\backslash H$ so that we have $C^\la(G,V) \cong \prod_{i\in I} C^\la(G_0 h_i,V)$ by the assumption $G= G_0 H$ and \Cref{Prop - Properties of locally analytic functions} (iii).
	One verifies that the product of the continuous homomorphisms
	\begin{equation*}
		\Ind_{H_0}^{\la,G_0}(V) \lra C^\la(G_0 h_i,V) \,,\quad f \lto \big[g \mto h_i^{-1} . f(g h_i^{-1})  \big] ,
	\end{equation*}
	yields an inverse to the above homomorphism. 
	It follows from \cite[Prop.\ 1.2 (i)]{SchneiderTeitelbaum02LocAnDistApplToGL2} that the closed subspace $\Ind_H^{\la, G} (V) \cong \Ind_{H_0}^{\la,G_0} (V)$ of $C^\la(G_0,V)$ (recall \Cref{Prop - Properties of locally analytic functions} (iv)) is of compact type as well.
\end{proof}

\begin{lemma}\label{Lemma - Tensoring up from distribution algebras of subgroups}
	Assume that there exists a compact open subgroup $G_0 \subset G$ such that $G= G_0 H$ and set $H_0 \defeq G_0 \cap H$.
	\begin{altenumerate}
		\item
			There is a topological isomorphism of $D(G_0)$-$D(H)$-bi-modules
			\begin{equation*}
				D(G_0) \ot{D(H_0),\iota} D(H) \overset{\cong}{\lra} D(G) \,,\quad \mu \otimes \nu \lto \mu \ast \nu .
			\end{equation*}
	
		\item
			For  $M \in \CM_{D(H)}$, the natural inclusion $D(G_0) \hookrightarrow D(G)$ induces a topological isomorphism of $D(G_0)$-modules
			\begin{equation*}
					D(G_0) \ot{D(H_0),\iota} M \overset{\cong}{\lra} D(G) \ot{D(H),\iota} M .
			\end{equation*}

	\end{altenumerate}
\end{lemma}
\begin{proof}
	On the level of abstract modules the statement of (i) is the content of \cite[Lemma 6.1 (i)]{SchneiderTeitelbaum05DualAdmLocAnRep}.
	But the reasoning there together with \cite[Lemma 1.2.13]{Kohlhaase05InvDistpAdicAnGrp} also shows the claimed topological isomorphism.

	For (ii), we obtain a continuous inverse as follows:
	Using (i) and \Cref{Lemma - Tensor identities for modules} (ii), there is a continuous homomorphism
	\begin{equation}\label{Eq - Homomorphism for tensoring up}
	\begin{aligned}
		D(G) \ot{K,\iota} M \cong D(G_0) \ot{D(H_0),\iota} \big( D(H) \ot{K,\iota} M \big) &\lra D(G_0) \ot{D(H_0),\iota} M , \\
			\delta \otimes \lambda \otimes m &\lto \delta \otimes \lambda \ast m .
	\end{aligned}
	\end{equation}
	It factors over the quotient $D(G) \ot{D(H),\iota} M$ and this yields the sought inverse.
\end{proof}

One can express the locally analytically induced representation $\Ind_{H}^{\la, G}(V)$ in terms of modules over $D(G)$.
For compact, $p$-adic $G$ and one-dimensional $V$ this is discussed before \cite[Lemma 6.1]{SchneiderTeitelbaum05DualAdmLocAnRep}.
Our proof of the following theorem is inspired by \textit{loc.\ cit.}.

\begin{theorem}\label{Prop - Module description for induction}
	Let $V$ be a locally analytic $H$-representation of compact type and assume that there exists a compact open subgroup $G_0 \subset G$ such that $G = G_0 H$.
	Then using the bilinear pairing of \Cref{Cor - Pairing of distributions and vector valued functions} there is a canonical topological isomorphism of $D(G)$-modules
	\begin{equation*}
		D(G) \cot{D(H),\iota} V'_b \overset{\cong}{\lra}  \big( \Ind_H^{\la,G} (V) \big)'_b \,,\quad \delta \otimes \ell \lto \big[ f \mto \ell\big( \delta(f) \big) \big] .
	\end{equation*}
\end{theorem}
\begin{proof}
	We fix an compact open subgroup $G_0 \subset G$ with $G=G_0 H$ and set $H_0 \defeq G_0 \cap H$.
	Using the bijective continuous integration homomorphism $I$ from \Cref{Thm - Integration map} together with \cite[Cor.\ 18.8]{Schneider02NonArchFunctAna} and the reflexivity of $C^\la(G,K)$ we obtain an injective continuous homomorphism
	\begin{equation*}
		\iota \colon \Ind_H^{\la, G}(V) \longhookrightarrow C^\la(G,V) \overset{I}{\lra} \CL_b(D(G),V)\cong C^\la(G,K) \cot{K,\pi} V.
	\end{equation*}
	This homomorphism is $G$-equivariant with respect to the left regular $G$-action on the source and the first factor of the target, and the trivial $G$-action on $V$.
	Analogously, we obtain $\iota_0 \colon \Ind_{H_0}^{\la,G_0}(V) \hookrightarrow C^\la(G_0,K) \cot{K} V$, where the map $C^\la (G_0,V) \ra C^\la(G_0,K) \cot{K}V$ in fact is the topological isomorphism from \Cref{Prop - Properties of locally analytic functions} (iv).
	Furthermore, recall from the proof of \Cref{Prop - Existence of locally analytic induction} that restriction of functions yields a topological isomorphism $\Ind_H^{\la, G} (V) \cong \Ind_{H_0}^{\la,G_0} (V)$.
	Restricting also induces a continuous homomorphism $C^\la (G,K) \cot{K,\pi} V \ra C^\la (G_0,K) \cot{K} V$.
	All of these maps fit into the commutative diagram
	\begin{equation}\label{Eq - Commutative diagram for induction}
		\begin{tikzcd}
			\Ind_{H_0}^{\la,G_0} (V) \ar[r, "\iota_0"] &C^\la (G_0,K) \cot{K} V \ar[r, "\psi_0"] &\prod_{g\in G_0, h\in 	H_0} V \\
			\Ind_H^{\la, G} (V) \ar[r, "\iota"]\ar[u, "\cong"] &C^\la (G,K) \cot{K,\pi} V \ar[r, "\psi"]\ar[u] &\prod_{g\in G, h\in H} V \ar[u] .
		\end{tikzcd}
	\end{equation}
	Here $\psi$ is the product of the homomorphisms
	\begin{equation*}
		C^\la (G,K) \cot{K,\pi} V \lra V \,,\quad f \otimes v \lto f(gh) v - f(g) h^{-1}. v ,
	\end{equation*}
	for $g\in G$, $h\in H$, and $\psi_0$ is defined analogously.
	The rows of \eqref{Eq - Commutative diagram for induction} are algebraically exact, and $\iota_0$ is a closed embedding.

	We want to consider the strong dual of the diagram \eqref{Eq - Commutative diagram for induction}.
	Note that by \cite[Rmk.\ 16.1 (ii)]{Schneider02NonArchFunctAna}, the homomorphisms of this dualized diagram are continuous again.
	Recall the topological isomorphism $\big( \prod_{i \in I} V_i \big)'_b \cong \bigoplus_{i\in I} (V_i)'_b$ from \cite[Prop.\ 9.11]{Schneider02NonArchFunctAna}, for a family $(V_i)_{i \in I}$ of locally convex $K$-vector spaces.
	Under this identification and the one from \Cref{Lemma - Dual of tensor product with locally analytic functions}, the transpose of $\psi$ is given by 
	\begin{align*}
		\transp{\psi} \colon \bigoplus_{g\in G, h\in H} V'_b \lra D(G) \cot{K,\iota} V'_b \,,\quad
		\sum \ell_{g,h} \lto \sum \delta_g \ast  \delta_h \otimes \ell_{g,h} - \delta_g \otimes h.\ell_{g,h}.
	\end{align*}
	We obtain an analogous description for $\psi_0$.
	The strong dual of \eqref{Eq - Commutative diagram for induction} becomes
	\begin{equation}
		\begin{tikzcd}
			\bigoplus_{g\in G_0, h\in H_0} V \ar[r, "\transp{\psi_0}"]\ar[d] & D(G_0) \cot{K} V'_b \ar[r, "\transp{\iota_0}"]\ar[d] & \big( \Ind_{H_0}^{\la,G_0} (V) \big)'_b \ar[d, "\cong"] \\
			\bigoplus_{g\in G, h\in H} V \ar[r, "\transp{\psi}"] & D(G) \cot{K,\iota} V'_b \ar[r, "\transp{\iota}"] & \big( \Ind_{H}^{\la,G} (V) \big)'_b.
		\end{tikzcd}
	\end{equation}
	Note that here both rows are complexes because the rows in \eqref{Eq - Commutative diagram for induction} are.
	The Hahn--Banach Theorem \cite[Cor.\ 9.4]{Schneider02NonArchFunctAna} implies that $\transp{\iota_0}$ is surjective because $\iota_0$ is a closed embedding.
	It follows from the open mapping theorem \cite[Prop.\ 8.6]{Schneider02NonArchFunctAna} that $\transp{\iota_0}$ is strict.
	Moreover, $\transp{\iota}$ and $\transp{\iota_0}$ are $D(G)$- resp.\ $D(G_0)$-linear.

	Since $\Im(\transp{\psi}) \subset \Ker(\transp{\iota})$ and the latter is closed, it follows from the density of Dirac distributions in $D(H)$ that all elements of $D(G) \cot{K,\iota} V'_b $ which are of the form
	\begin{equation*}
		\delta \ast \lambda \otimes \ell - \delta \otimes \lambda \ast \ell \quad\text{, for $\delta \in D(G)$, $\lambda \in D(H)$, $\ell \in V'_b$, }
	\end{equation*}
	are contained in $\Ker(\transp{\iota})$.
	Therefore the natural map $D(G) \ot{K,\iota} V'_b \ra D(G) \cot{K,\iota} V'_b$ induces a continuous $D(G)$-module homomorphism
	\begin{equation*}
		\phi \colon D(G) \ot{D(H),\iota} V'_b \lra \big( D(G) \cot{K,\iota} V'_b \big) \big/ \Ker(\transp{\iota}) .
	\end{equation*}
	The analogous statement also holds for $G_0$, and we obtain a commutative diagram
	\begin{equation*}
		\begin{tikzcd}
			D(G_0) \ot{D(H_0)} V'_b \ar[r, "\phi_0"]\ar[d, "\cong"'] &\big( D(G_0) \cot{K} V'_b \big) \big/ \Ker(\transp{\iota_0}) \ar[r, "\cong"]\ar[d] &\big( \Ind_{H_0}^{\la,G_0} (V) \big)'_b \ar[d, "\cong"] \\		
			D(G) \ot{D(H),\iota} V'_b \ar[r, "\phi"] &\big( D(G) \cot{K,\iota} V'_b \big) \big/ \Ker(\transp{\iota}) \ar[r, "\transp{\iota}"] &\big( \Ind_{H}^{\la,G} (V) \big)'_b .
		\end{tikzcd}
	\end{equation*}
	Here the left vertical map is a topological isomorphism by \Cref{Lemma - Tensoring up from distribution algebras of subgroups} (ii).

	We claim that after completing, $\transp{\iota} \circ \phi$ gives the sought topological isomorphism of $D(G)$-modules.
	For this it suffices to show that the completion of $\phi_0$ is a topological isomorphism.
	Using \Cref{Lemma - Completion of projective tensor product over algebras} the completion of $D(G_0) \ot{D(H_0)} V'_b$ is topologically isomorphic to the quotient of $D(G_0) \cot{K} V'_b$ by the closure of the $D(G_0)$-submodule generated by the elements
	\begin{equation*}
		\delta \ast \lambda \otimes \ell - \delta \otimes \lambda \ast \ell \quad\text{, for $\delta \in D(G_0)$, $\lambda \in D(H_0)$, $\ell \in V'_b$. }
	\end{equation*}
	By the density of Dirac distributions in $D(H_0)$ this closure is equal to $\ov{ \Im(\transp{\psi_0})}$.
	Thus it suffices to show that $\ov{ \Im(\transp{\psi_0}) }= \Ker(\transp{\iota_0}) $.
	But we have $\Ker(\transp{\iota_0}) = \Im(\iota_0)^\perp$ where
	\begin{equation*}
		\Im(\iota_0)^\perp \defeq \big\{ \ell \in D(G_0) \cot{K} V'_b \,\big\vert\, \forall v \in \Im(\iota_0): \ell(v) = 0 \big\} 
	\end{equation*}
	by \cite[p.IV.27 Prop.\ 2]{Bourbaki87TopVectSp1to5}.
	The algebraic exactness of the first row of \eqref{Eq - Commutative diagram for induction} implies $\Im(\iota_0)^\perp = \Ker(\psi_0)^\perp$. 
	Moreover, we have $\Ker(\psi_0)^\perp \subset \ov{\Im(\transp{\psi_0})}$ by \Cref{Lemma - Annihilator of kernel is weak closure of image of the transpose} and thus $\Ker(\transp{\iota_0}) \subset \ov{\Im(\transp{\psi_0})}$.
	As $\Im(\transp{\psi_0}) \subset \Ker(\transp{\iota_0})$ and $\Ker(\transp{\iota_0})$ is closed, we conclude that $\Ker(\transp{\iota_0}) = \ov{\Im(\transp{\psi_0})}$.
\end{proof}

\begin{corollary}[{cf.\ \cite[Prop.\ 5.1]{Kohlhaase11CohomLocAnRep}}]\label{Cor - Locally analytic induction is exact}
	Assume that there exists a compact open subgroup $G_0 \subset G$ such that $G = G_0 H$.
	The functor
	\begin{equation*}
		\Ind_H^{\la, G} \colon \Rep_K^{\la, \mathrm{ct}} (H) \lra  \Rep_K^{\la, \mathrm{ct}} (G) \,,\quad V \lto \Ind_H^{\la, G} (V) , 
	\end{equation*}
	between quasi-abelian categories is exact\footnote{In the sense of \cite[Def.\ 1.1.18]{Schneiders98QuasiAbCat}, i.e.\ it sends short strictly exact sequences to short strictly exact sequences.}.
\end{corollary}
\begin{proof}
	Because of the anti-equivalence from \Cref{Prop - Anti-equivalence for locally analytic representations} (ii) and \Cref{Prop - Module description for induction} the assertion is equivalent to the exactness of 
	\begin{equation*}
		\CM_{D(H)}^\mathrm{nF} \lra \CM_{D(G)}^\mathrm{nF} \,,\quad M \lto D(G) \cot{D(H),\iota} M .
	\end{equation*}
	The projection $G \ra G/H$ splits on the level of locally $L$-analytic manifolds and this yields an isomorphism of locally $L$-analytic manifolds $G\cong G/H \times H$, see \cite[Satz 4.1.1]{FeauxdeLacroix99TopDarstpAdischLieGrp}.
	Hence it follows from \Cref{Prop - Properties of space of distributions} (iii) that $D(G) \cong D(G/H) \cot{K,\iota} D(H)$.
	For $M \in \CM_{D(H)}^\mathrm{nF}$, \Cref{Lemma - Tensor identities for modules} yields a topological isomorphism
	\begin{equation*}
		D(G) \cot{D(H),\iota} M \cong D(G/H) \cot{K,\iota} M  .
	\end{equation*}
	Since $G/H$ is compact, the latter inductive tensor product agrees with the projective one.
	The exactness now follows from the exactness of the completed projective tensor product on $K$-Fr\'echet spaces, see \cite[Cor.\ 2.2]{BreuilHerzig18TowardsFinSlopePartGLn}.
\end{proof}

\section{The Hyperalgebra of a Non-Archimedean Lie Group}\label{Sect - The hyperalgebra}

\subsection{Germs of Locally Analytic Functions}

We recapitulate the concept of germs of locally analytic functions following F\'eaux de Lacroix \cite[Abschn.\ 2.3]{FeauxdeLacroix99TopDarstpAdischLieGrp}.
However, we will only need the case where the support consists of a single point.
In the following, $X$ is a locally $L$-analytic manifold and $V$ a Hausdorff locally convex $K$-vector space.

\begin{definition}\label{Def - Germs of locally analytic functions}
	For $x\in X$, the \textit{space of germs of locally analytic functions on $X$ with values in $V$ at $x$} is defined as the locally convex inductive limit over all open neighbourhoods $U\subset X$ of $x$
	\begin{equation*}
		C^\la_x (X,V) \defeq \varinjlim_{x\in U \subset X} C^\la (U,V)
	\end{equation*}
	with respect to the canonical restriction homomorphisms and endowed with the locally convex inductive limit topology.
\end{definition}

\begin{lemma}[{\cite[Abschn.\ 2.3.1]{FeauxdeLacroix99TopDarstpAdischLieGrp}}]\label{Lemma - Description of locally analytic germs}
	For every open subset $U\subset X$ with $x\in U$, the canonical map $C^\la(U,V) \ra C^\la_x(X,V) $ is a strict epimorphism.
	Moreover, there is a canonical topological isomorphism
	\begin{equation*}
		C^\la_x(X,V) \cong \varinjlim_{(U,E)} C^\an(U,E)
	\end{equation*}
	where the latter locally convex inductive limit is taken over all pairs of charts $\varphi\colon U\ra L^d$ with $x\in U$ and BH-subspaces $E\subset V$, and $C^\an(U,E)$ is defined as in \Cref{Def - Locally analytic functions} (i).
\end{lemma}

\begin{proposition}\label{Prop - Properties of locally analytic germs}
	\begin{altenumerate}
		\item 
			The formation of $C^\la_x(X,V)$ is contravariantly functorial in $(X,x)$ and covariantly functorial in $V$.
			
		\item
			For $x\in X$, the locally convex $K$-vector space $C^\la_x(X,V)$ is Hausdorff and barrelled.
			If $V$ is of compact type, then $C^\la_x(X,V)$ is of compact type and there is a natural topological isomorphism $C^\la_x(X,V) \cong C^\la_x(X,K) \cot{K} V$
			\textnormal{\cite[Satz 2.3.1, Satz 2.3.2]{FeauxdeLacroix99TopDarstpAdischLieGrp}}.
		\item
			For locally $L$-analytic manifolds $X$ and $Y$ with $x\in X$ and $y\in Y$, there is a topological isomorphism
			\begin{equation*}
				C^\la_x(X,K) \cot{K} C^\la_y (Y,K) \overset{\cong}{\lra} C^\la_{(x,y)}(X\times Y, K) \,,\quad f\otimes g \lto \big[ (z,w) \mto f(z)g(w) \big] .
			\end{equation*}
	\end{altenumerate}
\end{proposition}
\begin{proof}
	Using the functoriality of $C^\la(X,V)$ the statement of (i) is obvious.
	Concerning (iii), we may find a neighbourhood basis $(U_n)_{n\in \BN}$ of $x \in X$ resp.\  $(W_n)_{n\in \BN}$ of $y \in Y$.
	Then \Cref{Lemma - Description of locally analytic germs} shows that $C^\la_x(X,K) \cong \varinjlim_{n \in \BN} C^\an(U_n, K)$, similarly for $C^\la_y(Y,K)$, and 
	\begin{equation*}
		C^\la_{(x,y)}(X \times Y,K) \cong \varinjlim_{n \in \BN} C^\an(U_n \times W_n, K) \cong \varinjlim_{n \in \BN} \Big( C^\an(U_n,K) \cot{K} C^\an(W_n,K) \Big) .
	\end{equation*}
	The claim now follows from \cite[Prop.\ 1.1.32 (i)]{Emerton17LocAnVect}.
\end{proof}

\begin{proposition}\label{Prop - Properties of the algebra of germs of locally analytic functions}
	For $x\in X$, $C^\la_x(X,K)$ is a local, noetherian, jointly continuous locally convex $K$-algebra with respect to pointwise addition and multiplication. 
	Its maximal ideal is
	\begin{equation*}
		\Fm_x \defeq \Ker (\mathrm{ev}_x) \subset C^\la_x(X,K).
	\end{equation*}
	For all $n\in \BN_0$, the subspace $\Fm_x^{n+1} \subset C^\la_x(X,K)$ is closed and of finite codimension equal to $\binom{d+n}{n}$ where $d$ is the dimension of $X$ at $x$.
\end{proposition}
\begin{proof}
	It follows from \cite[Kor.\ 2.4.4]{FeauxdeLacroix99TopDarstpAdischLieGrp} that $C^\la_x(X,K)$ is a jointly continuous $K$-algebra.
	Clearly $\Fm_x$ is a maximal ideal, and via the inverse function theorem for power series, one shows that every $f\in C^\la_x(X,K)\setminus \Fm_x$ has a multiplicative inverse, i.e.\ that $C^\la_x(X,K)$ is local.

	Let $d\in \BN$ be the dimension of $X$ at $x$, and let $B^d_\ul{r}(0) \subset L^d$ denote the closed polydisc around $0$ of multiradius $\ul{r}\defeq (r,\ldots,r) \in \abs{K^\times}^d$.
	Then the choice of a chart around $x$ yields a topological isomorphism of $K$-algebras
	\begin{equation*}
		C^\la_x(X,K) \cong C^\la_0 (L^d,K) \cong \CA_0^d \defeq \varinjlim_{\ul{r}\in \abs{K^\times}^d} \CO \big( B_{\ul{r}}^d(0) \big). 
	\end{equation*}
	Here $\CA_0^d$ is noetherian by \cite[7.3.2 Prop.\ 7]{BoschGuentzerRemmert84NonArchAna}, and the above somorphism sends $\Fm_x$ to the ideal $(T_1,\ldots, T_d)$ where we view $\CA_0^d$ as a subalgebra of $K \llrrbracket{T_1,\ldots, T_d}$.
	Moreover, we have the strict epimorphism
	\begin{equation*}
		\CA_0^d \lra K^{\binom{d+n}{n}} \,,\quad \sum_{\ul{i}\in \BN^d_0} a_{\ul{i}} \, T_1^{i_1}\cdots T_d^{i_d} \lto (a_{\ul{i}})_{\sum_{j=1}^d i_j \leq n},
	\end{equation*}
	whose kernel is $(T_1,\ldots, T_d)^{n+1}$.
\end{proof}

\subsection{Distributions with Support in a Point}

Next we consider the strong dual space of $C^\la_x(X,K)$, i.e.\ distributions on $X$ with support in $x$ as treated by Kohlhaase in \cite[Sect.\ 1.2]{Kohlhaase05InvDistpAdicAnGrp}.

\begin{definition}\label{Def - Distributions with support}
	\begin{altenumerate}
		\item 
			For $x\in X$, we define the \emph{space of locally analytic distributions with support in $x$} as 
			\begin{equation*}
				D_x (X,K) \defeq C^\la_x(X,K)'_b .
			\end{equation*}
			Via the strict epimorphism $C^\la(X,K) \twoheadrightarrow C^\la_x(X,K)$ from \Cref{Lemma - Description of locally analytic germs}, it is naturally a closed subspace of $D(X,K)$ (cf.\ \cite[Lemma 1.2.5]{Kohlhaase05InvDistpAdicAnGrp}).
			
		\item
			For $n\in \BN_0$, we say that $\mu \in D_x(X,K)$ has \emph{order $\leq n$} if $\mu(\Fm_x^{n+1})= 0$.
			We denote the subspace of such distributions by $D_x(X,K)_n$.
			The \emph{subspace of distributions with support in $x$ of finite order} is defined as
			\begin{equation*}
				D_x (X,K)_\mathrm{fin} \defeq \bigcup_{n\in \BN_0} D_x(X,K)_n \subset D_x(X,K) .
			\end{equation*}
			We usually omit the coefficient field $K$ from the notation for these objects.
	\end{altenumerate}

\end{definition}

The definition of $D_x(X)_\mathrm{fin}$ has similarities with the space of distributions on an $L$-scheme with support in a rational point treated for example in \cite[Ch.\ I.7]{Jantzen03RepAlgGrp}.
We will discuss the precise relation of the two notions in \Cref{Sect - The Cases of Schemes and p-adic Lie Groups}.

\begin{lemma}\label{Lemma - Functoriality of hyperalgebra}
	Let $f\colon X \ra Y$ be a locally $L$-analytic map of locally $L$-analytic manifolds, and $x\in X$.
	Then, for all $n\in \BN_0$, there is a continuous homomorphism
	\begin{equation*}
		f_\ast \colon D_x(X)_n \lra D_{f(x)}(Y)_n \,,\quad \mu \lto \mu( \blank \circ f) .
	\end{equation*}
\end{lemma}
\begin{proof}
	The homomorphism $f^\ast \colon C^\la_{f(x)}(Y,K) \ra C^\la_x(X,K)$ induces the continuous homomorphism $f_\ast \colon D_x(X) \ra D_{f(x)}(Y)$, $\mu \mto \mu( \blank \circ f)$.
	Because $f^\ast$ is a homomorphism of local $K$-algebras, the claim follows.
\end{proof}

\begin{proposition}\label{Lemma - Hyperalgebra of product}
	Let $X$ and $Y$ be locally $L$-analytic manifolds with $x\in X$ and $y\in Y$.
	Then there is a canonical topological isomorphism
	\begin{equation*}
		D_{(x,y)}(X\times Y)_\mathrm{fin} \cong D_x(X)_\mathrm{fin} \ot{K} D_y(Y)_\mathrm{fin} 
	\end{equation*}
	which is compatible with the topological isomorphism $D(X\times Y) \cong D(X) \cot{K, \iota} D(Y)$ from \Cref{Prop - Properties of space of distributions} (iii).
\end{proposition}
\begin{proof}
	We begin by showing that under the topological isomorphism 
	\begin{equation*}\label{Eq - Germs of product of manifolds}
		\phi \colon C^\la_{(x,y)}(X\times Y, K)  \overset{\cong}{\lra} C^\la_x(X,K) \cot{K} C^\la_y (Y,K) 
	\end{equation*}
	of \Cref{Prop - Properties of locally analytic germs} (iii) we have $\phi \big( \Fm_{(x,y)}^k \big) = \sum_{j=0}^{k} \Fm_x^j \cot{K} \Fm_y^{k-j} $, for $k\in \BN_0$.
	For this purpose first note that $\ev_{(x,y)} = (\ev_x \otimes \ev_y) \circ \phi$.
	As the short strictly exact sequence
	\begin{equation*}
		0 \lra  \Fm_y  \lra  C^\la_y(Y,K) \lra K  \lra 0
	\end{equation*}	
	consists of locally convex $K$-vector spaces of compact type, its completed tensor product with $C^\la_x(X,K)$ remains strictly exact \cite[Cor.\ 2.2]{BreuilHerzig18TowardsFinSlopePartGLn}:
	\begin{equation*}
		0 \lra C^\la_x(X,K) \cot{K} \Fm_y  \lra  C^\la_x(X,K) \cot{K} C^\la_y(Y,K) \xrightarrow{\id \otimes \ev_y} C^\la_x(X,K)  \lra 0 .
	\end{equation*}
	Similarly, $\id \otimes \ev_y $ restricts to a strict epimorphism $ \Fm_x \cot{K} C^\la_y(Y,K) \twoheadrightarrow \Fm_x$.
	Since $\ev_x \otimes \ev_y = \mathrm{ev}_x \circ (\id \otimes \ev_y)$, we conclude that
	\begin{align*}
	 \Ker\big(\mathrm{ev}_x \otimes \mathrm{ev}_y \big) = \big(\id \otimes \ev_y \big)^{-1} (\Fm_x) 
		&= \Fm_x \cot{K} C^\la_y(Y,K) + \Ker \big( \id \otimes \ev_y \big)  \\
		&= \Fm_x \cot{K} C^\la_y(Y,K)  +C^\la_x(X,K) \cot{K} \Fm_y .
	\end{align*}
	The claim now follows since $\phi$ is an isomorphism of $K$-algebras.

	By \cite[Prop.\ 1.1.32 (ii)]{Emerton17LocAnVect} the transpose of $\phi$ gives a topological isomorphism
	\begin{equation*}
		D_x(X) \cot{K} D_y(Y) \overset{\cong}{\lra}	D_{(x,y)}(X\times Y).
	\end{equation*}
	The above claim implies that under this isomorphism $D_x(X)_m \ot{K} D_y(Y)_n$ is mapped into $D_{(x,y)}(X\times Y)_{n+m}$, for all $n,m \in \BN_0$.
	Hence we obtain a commutative diagram 
	\begin{equation*}
	\begin{tikzcd}
		D_x(X)_\mathrm{fin} \ot{K} D_y(Y)_\mathrm{fin} \ar[r, "\psi"] \ar[d,hook] & D_{(x,y)}(X\times Y)_\mathrm{fin}  \ar[d,hook] \\
		D_x(X) \cot{K} D_y(Y) \ar[r, "\cong"] & D_{(x,y)}(X\times Y)
	\end{tikzcd}
	\end{equation*}
	where the vertical maps are strict monomorphisms.
	Then \cite[Lemma 2]{KopylovKuzminov00KerCokerSeqSemiAbCat} implies that $\psi$ is a strict monomorphism as well.
	Thus it remains to prove that $\psi$ is surjective.
	To this end it suffices to show that its restriction
	\begin{equation*}
		\sum_{j=0}^k D_x(X)_j \ot{K} D_y(Y)_{k-j} \lra D_{(x,y)}(X\times Y)_{k}
	\end{equation*}
	is surjective for all $k\in \BN_0$.
	We do so by computing the dimension of both sides. 
	Let $d$ and $e$ be the dimension of $X$ at $x$ resp.\ of $Y$ at $y$.
	Writing $D_x(X)_{=j}$ for some complementary subspace of $D_x(X)_{j-1}$ in $D_x(X)_j$, the left hand side is the direct sum of $D_x(X)_{=j} \ot{K} D_y(Y)_{k-j}$, for $j=0,\ldots,k$.
	By \Cref{Lemma - Description of locally analytic germs} these direct summands are of dimension $\binom{d+j-1}{j}\binom{e+k-j}{k-j}$.
	Their dimensions add up to $\binom{d+e+k}{k}$, which agrees with the dimension of the right hand side.
\end{proof}

\begin{proposition}\label{Prop - Pairing of distributions at a point and vector valued germs}
	For $x\in X$, the pairing of \Cref{Cor - Pairing of distributions and vector valued functions} induces a natural, separately continuous, non-degenerate $K$-bilinear pairing
	\begin{equation*}
		D_x(X) \times C^\la_x(X,V) \lra V \,,\quad (\delta,f) \lto \delta(f) \defeq \delta(\tilde{f}),
	\end{equation*}
	where $\tilde{f} \in C^\la(X,V)$ denotes some lift of $f$.
	This pairing is compatible with the duality between $D_x(X)$ and $C^\la_x(X,K)$ in the sense that, for a BH-subspace $E \subset V$, its restriction to $D_x(X) \times C^\la_x(X,E)$ is given by the duality pairing $D_x(X) \times C^\la_x(X,K)\ra K$ tensored with $E$.
\end{proposition}
\begin{proof}
	To see that the above pairing is well defined, let $\tilde{f} \in \Ker\big( C^\la(X,V) \ra C^\la_x(X,V) \big)$.
	Let $\CI=(U_i,\varphi_i, E_i)_{i\in I}$ be some $V$-index on $X$ such that $\tilde{f} \in C^\la_\CI (X,V)$.
	The homomorphism $C^\la_\CI (X,V) \ra C^\la_x (X,V)$ factors over the inclusion $C^\an(U_{i_0}, E_{i_0}) \hookrightarrow C^\la_x(X,V)$ where $i_0\in I$ such that $x\in U_{i_0}$.
	It follows that $\tilde{f}\res{U_{i_0}} = 0$, and hence $\delta(\tilde{f}) = 0$ for all $\delta \in D(U_{i_0})$.
	Since $D_x(X) \subset D(U_{i_0})$, this shows the well-definedness.

	The separate continuity, the non-degeneracy in $D_x(X)$ and the compatibility statement of the above pairing follow from the respective statements of \Cref{Cor - Pairing of distributions and vector valued functions}.

	It remains to show the non-degeneracy in $C^\la_x(X,V)$.
	To this end, let $f\in C^\la_x(X,V)$ such that $\delta(f) = 0$ for all $\delta \in D_x(X)$.
	We have to show that $f = 0$.
	Let $U$ be a compact open chart around $x$ and $E\subset V$ a BH-subspace such that $f \in C^\an(U,E)\cong C^\an(U,K) \cot{K} \ov{E}$.
	We find bounded sequences $(f_n)_{n\in \BN}\subset C^\an(U,K)$ and $(v_n)_{n\in \BN} \subset \ov{E}$ such that $f = \sum_{n\geq 1} f_n \otimes v_n$.
	%Indeed, by the density and metrizability we find a sequence $v_n \in C^\an(U,K) \ot{K} \ov{E}$ such that $v_n \ra f$. Setting $w_0 \defeq 0$ and $w_k \defeq v_k -v_{k-1}$, for $k\in \BN$, we get a series $\sum_{k=0}^n w_k \ra f$ as $n\ra \infty$. Writing each $v_n$ as a sum of pure tensors, we can write this series as a series over a sequence of pure tensors. The convergence of this series implies that this sequence of pure tensors tends to $0$. As the tensor product norm of a pure tensor is the product of the norms of its constituents, we may by $K$-linear scaling of the constitutents of each pure tensor assume that both sequences making up the pure tensors are bounded.
	Since the homomorphism $D_x(X) \ra C^\an(U,K)'_b$ induced by the duality between $C^\la_x(X,K)$ and $D_x(X)$ via \Cref{Lemma - Description of locally analytic germs} is surjective, the assumption $\delta(f) = 0$ implies that 
	\begin{equation}\label{Eq - Functionals applied to tensor series}
		\sum_{n\geq 1} \lambda(f_n) v_n = 0 \quad \text{, for all $\lambda \in C^\an(U,K)'_b$.}
	\end{equation}
	But by \cite[Prop.\ 18.4]{Schneider02NonArchFunctAna} the $K$-Banach space $C^\an(U,K)$ satisfies the approximation property.
	Therefore \cite[Prop.\ 4.6]{Ryan02IntroTensProdBanachSp} shows that \eqref{Eq - Functionals applied to tensor series} implies $f = 0$.
\end{proof}

For $\mathrm{char}(L) = 0$ and a locally $L$-analytic Lie group $G$ with unit element $1$, F\'eaux de Lacroix shows that the pairing $U(\Fg)_K \times C^\la_1(G,V) \ra V$ is non-degenerate \cite[Kor.\ 4.7.4]{FeauxdeLacroix99TopDarstpAdischLieGrp}.
In view of \Cref{Prop - Hyperalgebra and universal enveloping algebra} below the following corollary generalizes this result.
F\'eaux de Lacroix's proof relies on differentiation with respect to elements of $\Fg$.
We pursue a different approach here which is more suitable for the case $\mathrm{char}(L) > 0$.

\begin{corollary}\label{Cor - Pairing of hyperalgebra and vector valued germs}
	For $x\in X$, $D_x(X)_\mathrm{fin} \subset D_x(X)$ is a dense subspace.
	In particular, the $K$-bilinear pairing
	\begin{equation*}
		D_x(X)_\mathrm{fin} \times C^\la_x(X,V) \lra V \,,\quad (\mu, f) \lto \mu(f) ,
	\end{equation*}
	induced from the one of \Cref{Prop - Pairing of distributions at a point and vector valued germs} is non-degenerate.
\end{corollary}
\begin{proof}
	As the pairing of \Cref{Prop - Pairing of distributions at a point and vector valued germs} is separately continuous, the second statement follows once we show that $D_x(X)_\mathrm{fin} \subset D_x(X)$ is a dense subspace.
	Assume, for the sake of contradiction, that there exists $\delta \in D_x(X) \setminus \ov{ D_x(X)_\mathrm{fin} }$.
	By the Hahn--Banach theorem \cite[Cor.\ 9.3]{Schneider02NonArchFunctAna} we find $f\in C^\la_x(X,K)$ such that $\mu(f) =0$, for all $\mu \in \ov{D_x(X)_\mathrm{fin} }$, and $\delta(f) =1$.
	For any $n\in \BN_0$, it follows that $f\in D_x(X)_n^\perp$ where we consider the orthogonal under the pairing of \Cref{Prop - Pairing of distributions at a point and vector valued germs}.
	But we have $D_x(X)_n = (\Fm_x^{n+1} )^\perp$ by definition, and $\big( (\Fm_x^{n+1})^\perp \big)^\perp = \Fm_x^{n+1}$ by \cite[p.II.45 Cor.\ 3]{Bourbaki87TopVectSp1to5}.
	Hence $f\in \bigcap_{n\in \BN_0} \Fm_x^{n+1}$, and we apply Krull's intersection theorem to the local, noetherian $K$-algebra $C^\la_x(X,K)$ to conclude that $f=0$, a contradiction.
\end{proof}

\subsection{The Hyperalgebra of a Locally Analytic Lie Group}

Let $G$ again be a locally $L$-analytic Lie group with multiplication map $m$ and identity element $1$.

\begin{definition}\label{Def - Hyperalgebra}
	We define the \emph{hyperalgebra\footnote{Despite the analogy of $\hy(G)$ with the distribution algebra of an algebraic group (cf.\ \Cref{Sect - The Cases of Schemes and p-adic Lie Groups}) we prefer the name ``hyperalgebra'' here to distinguish it from the algebra of locally analytic distributions on $G$.} of $G$ (over $K$)} as
	\begin{equation*}
		\hy(G,K) \defeq D_1(G,K)_\mathrm{fin} .
	\end{equation*}
	It is the union of the finite-dimensional subspaces $\hy(G,K)_n \defeq D_1(G,K)_n$, for $n\in \BN_0$.
	We omit mention of the coefficient field $K$ from the notation when it does not cause confusion.
\end{definition}

As the next lemma demonstrates, $\hy(G)$ is a jointly continuous locally convex subalgebra of $D(G)$ with respect to the convolution product.
Note that the multiplication map of $\hy(G)$ is jointly continuous since $\hy(G) \subset D(G_0)$, for any compact open subgroup $G_0$ of $G$.
By \Cref{Lemma - Functoriality of hyperalgebra}, $\hy(G)$ is preserved under the involution $\mu \mto \dot{\mu}$.

\begin{lemma}\label{Lemma - Multiplication for hyperalgebra}
	For $n, m \in \BN_0$, we have $	\hy(G)_n \ast \hy(G)_m \subset \hy(G)_{n+m} $.
\end{lemma}
\begin{proof}
	In the proof of \Cref{Lemma - Hyperalgebra of product} we have seen that under the topological isomorphism $D(G) \cot{K,\iota} D(G) \cong D(G\times G)$ the subspace $\hy(G)_n \ot{K} \hy(G)_m$ is mapped into the subspace $D_1(G\times G)_{n+m}$.
	By \Cref{Lemma - Functoriality of hyperalgebra} the image of the latter under $m_\ast \colon D(G \times G) \ra D(G)$ is contained in $\hy(G)_{n+m}$.
\end{proof}

We obtain an analogue of the \emph{adjoint representation} of $G$ on $\hy(G)$ where $g\in G$ acts by the algebra homomorphism 
\begin{equation*}
	\Ad(g) \colon  \hy(G) \lra \hy(G) \,,\quad \mu \lto \Ad(g)(\mu) \defeq \left[f \mto \mu\big(f(g\blank g^{-1})\big)\right].
\end{equation*}

\begin{proposition}\label{Prop - Adjoint representation}
	The adjoint representation $\Ad$ of $G$ on $\hy(G)$ is locally analytic (or equivalently extends to a separately continuous $D(G)$-module structure on $\hy(G)$).
	For every $n\in \BN_0$, it restricts to a locally analytic $G$-subrepresentation $\Ad_n$ on $\hy(G)_n \subset \hy(G)$.
\end{proposition}
\begin{proof}
	We first want to show that the $G$-representation 
	\begin{equation}\label{Eq - Representation by conjugation on stalk}
		G \times C^\la_1 (G,K) \lra C^\la_1 (G,K) \,,\quad (g,f) \lto f(g^{-1} \blank g),
	\end{equation}
	is locally analytic. 
	It follows from functoriality (\Cref{Prop - Properties of locally analytic germs} (i)) that this is an action of $G$ on $C^\la_1(G,V)$ by topological automorphisms.
	Moreover, for some compact open subgroup $G_0$ of $G$, the $G_0$-representation on $C^\la(G_0,V)$ by conjugation is  locally analytic, see \Cref{Expl - Regular representations}.
	Since the quotient map $C^\la(G_0,V) \twoheadrightarrow C^\la_1(G,V)$ is $G_0$-equivariant, the latter is a locally analytic $G_0$-representation by \Cref{Prop - Subrepresentations and quotients of locally analytic representations} (i).
	Therefore the $G$-representation \eqref{Eq - Representation by conjugation on stalk} is locally analytic by \cite[Prop.\ 3.6.11]{Emerton17LocAnVect}.

	Note that $C^\la_1(G,K)$ is of compact type (\Cref{Prop - Properties of locally analytic germs} (ii)).
	Via the anti-equivalence of \Cref{Prop - Anti-equivalence for locally analytic representations} (ii), $D_1(G)$ becomes a separately continuous $D(G)$-module.
	The action of $\delta_g$, for $g\in G$, is given by 
	\begin{equation*}
		\delta_g \ast \mu = \left[f \mto \mu\big(f(g\blank g^{-1})\big)\right] \quad\text{, for $\mu \in D_1(G)$.}
	\end{equation*}
	Moreover, $\Fm_1^{n+1} \subset C^\la_1(G,K)$ is a closed subspace which is invariant under the $G$-action of \eqref{Eq - Representation by conjugation on stalk}.
	It follows that $\hy(G)_n = \big( C^\la_1(G,K) / \Fm_1^{n+1} \big)^\perp$ and consequently also $\hy(G)$ are $D(G)$-submodules of $D_1(G)$.
	Finally \Cref{Prop - Anti-equivalence for locally analytic representations} (i) implies that $\Ad$ and $\Ad_n$ are locally analytic $G$-representation because $\hy(G)$ is of LB-type.
\end{proof}

\begin{lemma}\label{Lemma - Formula for adjoint representation}
	For $\delta \in D(G)$ and $\mu \in \hy(G)$, we have $\Ad(\delta)(\mu) \ast \delta = \delta \ast \mu$.
\end{lemma}
\begin{proof}
	It suffices to consider the case of a Dirac distribution $\delta = \delta_g$, for $g\in G$.
	Using \Cref{Lemma - Fubinis theorem} we compute, for $f\in C^\la(G,K)$,
	\begin{align*}
		\big( \Ad(g)(\mu) \ast \delta_g \big)(f) &= \Ad(g)(\mu)\big[ h \mto f(hg) \big] =\mu\big( f(g\blank g^{-1} g) \big)  \\
			&= \mu \big[ h' \mto f(gh') \big] = \big( \delta_g \ast \mu \big)(f) .
	\end{align*}
\end{proof}

\subsection{The Cases of Schemes and $p$-adic Lie Groups}\label{Sect - The Cases of Schemes and p-adic Lie Groups}

As mentioned before, we want to compare $\hy(G)$ to a certain distribution algebra when the non-archimedean Lie group $G$ arises from a smooth algebraic group.
To recall this notion, let $\bX$ be an $L$-scheme with an $L$-valued point $x \in \bX(L)$, and let $\FM_x \subset \CO_{\bX,x}$ denote the maximal ideal of the stalk of the structure sheaf $\CO_\bX$ at $x$.

\begin{definition}[{\cite[Sect.\ I.7.1]{Jantzen03RepAlgGrp}}]
	For $n\in \BN_0$, the $L$-vector space of \emph{distributions on $\bX$ of order $\leq n$ with support in $x$} is defined to be
	\begin{equation*}
		\mathrm{Dist}(\bX,x)_n \defeq \left\{\mu \in (\CO_{\bX,x})^\ast \middle{|} \mu(\FM_x^{n+1}) = 0 \right\}.
	\end{equation*}
	The space $\mathrm{Dist}(\bX,x)$ of \emph{distributions on $\bX$ with support in $x$} is the union of all $\mathrm{Dist}(\bX,x)_n$.
\end{definition}

Let $\bX$ be of finite type over $L$ and let $\bX^\rig$ denote the rigid analytic space over $L$ associated to $\bX$, cf.\ \cite[9.3.4 Expl.\ 2]{BoschGuentzerRemmert84NonArchAna}.
We write $\Fm_x \subset \CO_{\bX^\rig,x}$ for the maximal ideal in the rigid analytic stalk, cf.\  \cite[Sect.\ 7.3.2]{BoschGuentzerRemmert84NonArchAna}.
Then one has a canonical inclusion $\CO_{\bX,x} \hookrightarrow \CO_{\bX^\rig,x}$ of local $L$-algebras.
It induces an isomorphism between the completions with respect to the respective maximal ideals as stated in \cite[p.113]{Bosch14LectFormRigidGeom}.
Passing to the algebraic duals then yields
\begin{equation*}
	\mathrm{Dist}(\bX,x)_n \cong \big( \CO_{\bX,x}/\FM_x^{n+1} \big)^\ast \cong \big( \CO_{\bX^\rig,x}/\Fm_x^{n+1} \big)^\ast .
\end{equation*}
On the other hand, if in addition $\bX$ is smooth and separated, the set of $L$-valued points $\bX(L) = \bX^\rig(L)$ naturally carries the structure of a locally $L$-analytic manifold.
The description of \Cref{Lemma - Description of locally analytic germs} then yields an isomorphism of local $K$-algebras $(\CO_{\bX^\rig,x} )_K  \cong C^\la_x(\bX(L),K)$.
For the dual spaces we obtain
\begin{equation*}
	D_x(\bX(L),K)_n \cong \big( C^\la_x(\bX(L),K)/\Fm_x^{n+1} \big)^\ast
		\cong \big( (\CO_{\bX^\rig,x} )_K / (\Fm_x^{n+1})_K \big)^\ast 
		\cong \big( \CO_{\bX^\rig,x}/\Fm_x^{n+1} \big)^\ast_K .
\end{equation*}
Note that every $\mu \in C^\la_x(\bX(L),K)^\ast $ with $\mu(\Fm_x^{n+1})=0$ here is continuous automatically.
In total we obtain:

\begin{proposition}\label{Prop - Distributions on scheme}
	Let $\bX$ be a smooth, separated $L$-scheme of finite type and $x\in \bX(L)$.
	Then the inclusion $(\CO_{\bX,x} )_K \hookrightarrow C^\la_x(\bX(L),K)$ induces a canonical isomorphism of $K$-vector spaces
	\begin{equation*}
		D_x(\bX(L),K)_\mathrm{fin}  \overset{\cong}{\lra} \mathrm{Dist}(\bX,x)_K .
	\end{equation*} 
\end{proposition}

Now let $\bG$ be a (smooth) linear algebraic group over $L$ with multiplication morphism $m$ and identity element $1 \in \bG(L)$.
Then $\bG(L)$ has the structure of a locally $L$-analytic Lie group.
The space ${\rm Dist}(\bG,1) $ carries an algebra structure where the product of $\mu$ and $\nu$ is given by \cite[Sect.\ I.7.7]{Jantzen03RepAlgGrp}
\begin{equation*}
	\CO_{\bG,1}  \xrightarrow{m^\ast}  \CO_{\bG,1} \ot{L}  \CO_{\bG,1}  \xrightarrow{\mu \otimes \nu} L .
\end{equation*}
Since the above map is compatible with \eqref{Eq - Explicit map for convolution of distributions} under the isomorphism of \Cref{Prop - Distributions on scheme}, we can strengthen the latter.

\begin{corollary}\label{Cor - Distributions of algebraic group}
	For a smooth linear algebraic group $\bG$ over $L$, there is a canonical isomorphism of $K$-algebras
	\begin{equation*}
		\hy(\bG(L),K) \overset{\cong}{\lra}  {\rm Dist}(\bG,1)_K .
	\end{equation*}
\end{corollary}

\medskip

Now assume that $L$ is a $p$-adic field.
To a locally $L$-analytic Lie group $G$, one can associate its Lie algebra $\Fg$, see \cite[III.\ \S 3.7]{Bourbaki89LieGrpLieAlg1to3} or \cite[Sect.\ 13]{Schneider11pAdicLieGrps}.
As $\mathrm{char}(L)=0$, there exists an exponential map $\exp \colon \Fg \dashrightarrow G$ defined locally around $0\in \Fg$ \cite[III.4.3 Def.\ 1]{Bourbaki89LieGrpLieAlg1to3}.
Schneider and Teitelbaum use this to embed $\Fg$ into $D(G,K)$ via the rule
\begin{equation}\label{Eq - Formula for applying Lie algebra element}
	\Fx (f) \defeq \frac{d}{dt} f(\exp(t\Fx))\res{t=0} \quad\text{, for $\Fx \in \Fg$, $f \in C^\la(G,K)$.}
\end{equation}
In fact this embedding extends to an algebra homomorphism \cite[Sect.\ 2]{SchneiderTeitelbaum02LocAnDistApplToGL2}.
\begin{equation*}
	\iota \colon U(\Fg)_K \longhookrightarrow D(G,K) .
\end{equation*}
The following proposition allows us to identify $U(\Fg)_K$ with $\hy(G,K)$.

\begin{proposition}\label{Prop - Hyperalgebra and universal enveloping algebra}
	For a $p$-adic Lie group $G$, we have $\Im(\iota) = \hy(G,K)$.
\end{proposition}
\begin{proof}
	The expression of \eqref{Eq - Formula for applying Lie algebra element} vanishes if $f$ is the zero function in a neighbourhood of $1$.
	Therefore the image of $\iota$ is contained in $D_1(G,K)$.
	Furthermore, using the product rule one verifies that $\iota(\Fx)(\Fm_1^2)=0$, for all $\Fx \in \Fg$, and hence $\iota(\Fg_K) \subset \hy(G,K)_1$.
	Let $U_n(\Fg)_K \subset U(\Fg)_K$ be the subspace spanned by all products $\Fx_1\cdots \Fx_r$, for $r\leq n$, $\Fx_i \in \Fg$.
	It follows that $U_n(\Fg)_K$ is mapped into $\hy(G,K)_n$ under $\iota$, and hence $\iota(U(\Fg)_K)\subset \hy(G,K)$.
	We deduce equality from $\dim_K ( U_n(\Fg)_K )= \dim_K (\hy(G,K)_n )$ for all $n\in \BN_0$.	
\end{proof}

\section{The Functors $\dot{\CF}_P^G$}\label{Sect - The functors FGP}

Throughout this section, let $K$ be a finite extension of the non-archimedean local field $L$.
Let $G$ denote a locally $L$-analytic Lie group with identity element $1$ and $H\subset G$ a locally $L$-analytic Lie subgroup.
Furthermore, we assume there exists a compact open subgroup $G_0$ of $G$ such that $G = G_0 H$, and write $H_0 \defeq G_0 \cap H$.

An important instance of this situation is $G = \bG(L)$ and $H =  \bP(L)$ for a connected reductive algebraic group $\bG$ over $L$ and a parabolic subgroup $\bP$ of $\bG$.
For $G_0$ one may take a special, good, maximal compact subgroup of $G$ from Bruhat--Tits theory so that $\bG(L)= G_0 \bP(L)$ is a consequence of the Iwasawa decomposition, see \cite[Sect.\ 3.5]{Cartier79ReppAdicGrpsSurvey}.

\subsection{The Subalgebras $D(\dot{\Fg}, H)$}

In analogy to the definition of Orlik and Strauch \cite[Sect.\ 3.4]{OrlikStrauch15JordanHoelderSerLocAnRep} we consider the separately continuous $K$-subalgebra
\begin{equation*}
	D(\dot{\Fg},H,K) \defeq \hy(G,K) \cdot D(H,K) \subset D(G,K)
\end{equation*}
generated by $\hy(G,K)$ and $D(H,K)$.
Again, we usually omit the coefficient field $K$ in this context.
The description of \cite[Prop.\ 3.5]{OrlikStrauch15JordanHoelderSerLocAnRep} and its proof generalize:

\begin{proposition}\label{Prop - Description of composite subalgebra}
	Every element of $D(\dot{\Fg},H)$ is a finite sum of elements of the form $\mu \ast \delta$, for $\mu \in \hy(G)$, $\delta \in D(H)$.
\end{proposition}
\begin{proof}
	It suffices to consider elements of the form $\delta \ast \mu$, for $\mu \in \hy(G)$, $\delta \in D(H)$.
	Fix such an element and $n\in \BN_0$ such that $\mu \in \hy(G)_n$.
	The adjoint representation $\Ad_n$ on $\hy(G)_n$ is given by a locally $L$-analytic map of locally $L$-analytic Lie groups $G \ra \GL\big(\hy(G)_n \big)$ by \Cref{Cor - Locally analytic Banach space representation is given by homomorphism of Lie groups}.
	Hence, for an $L$-Basis $\mu_1,\ldots,\mu_r$ of $\hy(G)_n$, there exist $c_1,\ldots,c_r \in C^\la(G,K)$ such that
	\begin{equation*}
		\Ad_n(g)(\mu) = \sum_{i=1}^r c_i(g) \, \mu_i \quad\text{, for all $g\in G$.}
	\end{equation*}
	We define $\delta_i \in D(H)$, for $i=1,\ldots,r$, by the rule
	\begin{equation*}
		\delta_i(f) \defeq \delta\big[ h \mto c_i(h) \, f(h) \big] \quad \text{, for $f\in C^\la(H,K)$.} 
	\end{equation*}
	Analogously to the proof of \cite[Prop.\ 3.5]{OrlikStrauch15JordanHoelderSerLocAnRep} (and using the respective ``Fubini's theorem'' \Cref{Lemma - Fubinis theorem})) one computes that $\delta \ast \mu = \sum_{i=1}^r \mu_i \ast \delta_i$.
\end{proof}

\begin{corollary}\label{Cor - Compatibility with adjoint representation}
	For any locally analytic $G$-representation $V$, we have
	\begin{equation*}
		\delta \ast (\mu \ast v) = \Ad(\delta)(\mu) \ast (\delta \ast v) \quad\text{, for all $\delta \in D(G)$, $\mu \in \hy(G)$, $v \in V$.}
	\end{equation*}
\end{corollary}
\begin{proof}
	By density it suffices to consider a Dirac distribution $\delta = \delta_g$, for $g\in G$. 
	Let $\mu \in \hy(G)_n$, for some $n\in \BN_0$.
	We find $\mu_1,\ldots,\mu_r \in \hy(G)_n$ and $c_1,\ldots,c_r \in C^\la(G,K)$ as in the proof of the above proposition.
	For $H\defeq G$ and $\delta \defeq \delta_g$ in that proof, we have $\delta_i = c_i(g)\delta_g$ and consequently $\delta_g \ast \mu = \sum_{i=1}^r c_i(g) \, \mu_i \ast  \delta_g $.
	We then compute, for $v\in V$, that
	\begin{align*}
		\Ad_n(g)(\mu) \ast (\delta_g \ast v) 
			= \sum_{i=1}^r c_i(g) \, \mu_i \ast (\delta_g \ast v) 
			= \bigg( \sum_{i=1}^r c_i(g) \, \mu_i \ast \delta_g \bigg) \ast v 
			= ( \delta_g \ast \mu ) \ast v.
	\end{align*}
\end{proof}

The following statement for $p$-adic $L$ and on the level of abstract $K$-algebras is part of \cite[Lemma 4.1]{SchmidtStrauch16DimLocAnRep}.

\begin{corollary}\label{Cor - Isomorphism for composite algebra}
	The multiplication in $D(G)$ induces a topological isomorphism
	\begin{equation*}
		\hy(G) \ot{\hy(H),\iota} D(H) \overset{\cong}{\lra} D(\dot{\Fg},H) \,,\quad \mu \otimes \delta \lto \mu \ast \delta .
	\end{equation*}
\end{corollary}
\begin{proof}
	The multiplication homomorphism of $D(G)$ induces the continuous homomorphism
	\begin{equation*}
		\hy(G) \ot{\hy(H),\iota} D(H) \lra D(G).
	\end{equation*}
	By \Cref{Prop - Description of composite subalgebra} the image of this homomorphism equals $D(\dot{\Fg},H)$.
	On the other hand, the canonical projection $G \ra G/H$ splits on the level of locally $L$-analytic manifolds and such a splitting yields an isomorphism of locally $L$-analytic manifolds $G \cong G/H \times H$, see \cite[Satz 4.1.1]{FeauxdeLacroix99TopDarstpAdischLieGrp}.
	From \Cref{Lemma - Hyperalgebra of product} we thus obtain a topological isomorphism $\hy(G) \cong D_{1H}(G/H)_\mathrm{fin} \ot{K,\iota} \hy(H)$.
	Using \Cref{Lemma - Tensor identities for modules} we obtain the commutative diagram
	\begin{equation*}
	\begin{tikzcd}
			\hy(G) \ot{\hy(H),\iota} D(H) \ar[r] \ar[d, "\cong"] & D(G) \ar[d, "\cong"] \\
			D_{1H}(G/H)_\mathrm{fin} \ot{K,\iota} D(H) \ar[r] & D(G/H) \cot{K,\iota} D(H) 
	\end{tikzcd}
	\end{equation*}
	where the vertical maps are topological isomorphisms.
	Once we show that the bottom map is an embedding, it follows from \cite[Lemma 2]{KopylovKuzminov00KerCokerSeqSemiAbCat} that the top map is a topological isomorphism onto its image $D(\dot{\Fg},H)$.	
	But the bottom map is the composite 
	\begin{equation*}
		D_{1H}(G/H)_\mathrm{fin} \ot{K,\iota} D(H)\lra D(G/H) \ot{K,\iota} D(H) \longhookrightarrow D(G/H) \cot{K,\iota} D(H) .
	\end{equation*}
	Using \cite[Lemma 1.2.13]{Kohlhaase05InvDistpAdicAnGrp} the first morphism is the direct sum of 
	\begin{equation*}
		D_{1H}(G/H)_\mathrm{fin} \ot{K,\iota} D(U_i) \lra D(G/H) \ot{K,\iota} D(U_i) ,
	\end{equation*}
	for some disjoint covering $H = \bigcup_{i\in I} U_i$ by open compact subsets.
	The claim now follows from \Cref{Lemma - Equality of projective and inductive tensor product} and the exactness of the projective tensor product \cite[Lemma 2.1 (ii)]{BreuilHerzig18TowardsFinSlopePartGLn}.
\end{proof}

\subsection{Separately Continuous $D(\dot{\Fg},H)$-Modules}

We will now study the category $\CM_{D(\dot{\Fg},H)}$ of separately continuous $D(\dot{\Fg},H)$-modules as well as its full additive subcategory $\CM_{D(\dot{\Fg},H)}^\mathrm{nF}$ whose $D(\dot{\Fg},H)$-modules are nuclear $K$-Fr\'echet spaces.

\begin{lemma}\label{Lemma - Module structure on tensor product}
	For $M, N \in \CM_{D(\dot{\Fg},H)}$, the inductive tensor product $M\ot{K,\iota} N$ naturally becomes a separately continuous $D(\dot{\Fg},H)$-module via
	\begin{equation*}
		D(\dot{\Fg},H) \times M\ot{K,\iota} N \lra M\ot{K,\iota} N \,,\quad (\delta, m \otimes n ) \lto \delta \ast m \otimes \delta \ast n .
	\end{equation*}
	The analogous assertion for separately continuous $D(G)$-modules holds as well.
\end{lemma}
\begin{proof}
	We first consider the case of separately continuous $D(G)$-modules.
	The only assertion which is not verified immediately is the separate continuity of the multiplication map.
	To this end, it suffices to prove that
	\begin{equation}\label{Eq - Diagonal map for distribution algebra}
		D(G) \lra D(G) \ot{K,\iota} D(G) \,,\quad \delta \lto \delta \otimes \delta ,
	\end{equation}
	is continuous.
	Let $G = \bigcup_{i\in I} U_i$ be some disjoint covering by compact open subsets.
	Then \Cref{Prop - Properties of space of distributions} (iii) and \cite[Lemma 1.2.13]{Kohlhaase05InvDistpAdicAnGrp} yield topological isomorphisms
	\begin{equation*}
		D(G) \cong \bigoplus_{i\in I} D(U_i)  \quad\text{ and }\quad D(G) \ot{K,\iota} D(G) \cong \bigoplus_{i,j\in I} D(U_i) \ot{K,\iota} D(U_j) ,
	\end{equation*}
	and $D(U_i)$ maps into $D(U_i) \ot{K,\iota} D(U_i)$ under \eqref{Eq - Diagonal map for distribution algebra}.
	But $D(U_i)$ is a $K$-Fr\'echet space so that $D(U_i) \ot{K,\iota} D(U_i) = D(U_i) \ot{K,\pi} D(U_i)$. 
	Since \eqref{Eq - Diagonal map for distribution algebra} is the composition of the diagonal embedding and the canonical continuous map $D(U_i) \times D(U_i) \ra D(U_i) \ot{K,\pi} D(U_i)$ (cf.\ \cite[Sect.\ 17.B]{Schneider02NonArchFunctAna}), its continuity follows.

	Furthermore, then the homomorphism $D(\dot{\Fg},H) \ra D(\dot{\Fg},H) \ot{K,\iota} D(\dot{\Fg},H)$ induced by \eqref{Eq - Diagonal map for distribution algebra} is continuous.
	This implies the statement for separately continuous $D(\dot{\Fg},H)$-modules.
\end{proof}

\begin{proposition}\label{Prop - Extension of module structure to completion}
	Let $M \in \CM_{D(\dot{\Fg},H)}$ and assume that the underlying locally convex $K$-vector space of $M$ is pseudo-metrizable and barrelled.
	Then the module structure on $M$ extends uniquely to a separately continuous $D(\dot{\Fg},H)$-module structure on the Hausdorff completion $\widehat{M}$.
	In particular, if in addition $M$ is nuclear, then $\widehat{M} \in \CM_{D(\dot{\Fg},H)}^\mathrm{nF}$.
\end{proposition}
\begin{proof}
	We first consider the compact open subgroup $H_0 \subset H$.
	Since $D(H_0)$ is a $K$-Fr\'echet space and $\hy(G)$ a subspace of the $K$-Fr\'echet space $D(G_0)$, it follows from \Cref{Lemma - Properties of pseudo-metrizable spaces} (i) that $D(\Fg,H_0) \cong \hy(G) \ot{\hy(H)} D(H_0)$ is pseudo-metrizable.
	Hence \Cref{Lemma - Equality of projective and inductive tensor product} implies that the scalar multiplication of $D(\Fg, H_0)$ with $M$ is jointly continuous.
	Therefore it extends to a continuous homomorphism
	\begin{equation}\label{Eq - Extension of scalar mulitplication to completion}
		D(\Fg,H_0) \ot{K} \widehat{M} \lra \widehat{M}
	\end{equation}
	by \cite[Lemma 1.2.2]{Emerton17LocAnVect}.
	Using $D(H) \cong \bigoplus_{h \in H_0\backslash H} D(H_0) \ast \delta_h$ and \cite[Lemma 1.2.13]{Kohlhaase05InvDistpAdicAnGrp}, we find that $D(\Fg,H) \cong \bigoplus_{h \in H_0\backslash H} D(\Fg,H_0) \ast \delta_h$.
	Hence the direct sum
	\begin{equation*}
		D(\Fg,H) \ot{K,\iota} \widehat{M} \cong \bigoplus_{h \in H_0\backslash H} D(\Fg,H_0) \ast \delta_h \ot{K} \widehat{M} \lra \widehat{M}
	\end{equation*}
	of the ``translates'' of the homomorphism \eqref{Eq - Extension of scalar mulitplication to completion} yields the sought extension to a separately continuous $D(\dot{\Fg},H)$-module structure on $\widehat{M}$.

	Finally note that the Hausdorff completion of a nuclear pseudo-metrizable locally convex $K$-vector space is a nuclear $K$-Fr\'echet space, see \Cref{Lemma - Properties of pseudo-metrizable spaces} (iv) and \cite[Prop.\ 20.4]{Schneider02NonArchFunctAna}.
\end{proof}

\begin{corollary}\label{Cor - Completed tensor product is module}
	The quasi-abelian categories $\CM_{D(\dot{\Fg},H)}^\mathrm{nF}$ and $\CM_{D(G)}^\mathrm{nF}$ (see \Cref{Prop - Category of nuclear Frechet modules is quasi-abelian}) are closed under taking the completed tensor product over $K$.
\end{corollary}
\begin{proof}
	Note that, for nuclear $K$-Fr\'echet spaces $M$ and $N$, $M \ot{K} N$ is pseudo-metrizable by \Cref{Lemma - Properties of pseudo-metrizable spaces} (i) and nuclear by \cite[Prop.\ 19.11]{Schneider02NonArchFunctAna}.
	Furthermore, since $M \ot{K,\pi} N = M \ot{K,\iota} N$ carries the locally convex final topology with respect to a family of maps from barrelled spaces, it is barrelled itself \cite[p.35, Expl.\ 3]{Schneider02NonArchFunctAna}.
	The claim then follows from \Cref{Lemma - Module structure on tensor product} and \Cref{Prop - Extension of module structure to completion}.
\end{proof}

We can also characterize separately continuous $D(\dot{\Fg},H)$-modules analogously to the $p$-adic situation considered by Agrawal and Strauch in \cite[Sect.\ 7.4]{AgrawalStrauch22FromCatOLocAnRep}.

\begin{definition}\label{Def - Compatible hyperalgebra modules}
	A \emph{locally analytic $(\hy(G),H)$-module} is a locally analytic $H$-representation $V$ which simultaneously is a separately continuous $\hy(G)$-module such that the following two compatibly conditions hold:
	\begin{altenumeratelevel2}
		\item
			The action of $\hy(H)$ as a $K$-subalgebra of $\hy(G)$ agrees with the action induced from \Cref{Prop - Anti-equivalence for locally analytic representations} (i) of $\hy(H)$ as a $K$-subalgebra of $D(H)$.
		
		\item
			For all $h \in H$, $\mu\in \hy(G)$ and $v \in V$, we have 
			\begin{equation*}
				h. (\mu \ast v) = \Ad(h)(\mu) \ast (h.v) .
			\end{equation*}
	\end{altenumeratelevel2}
\end{definition}

\begin{remark}\label{Rmk - Locally analytic representation is compatible hyperalgebra module}
	It follows from \Cref{Cor - Compatibility with adjoint representation} that every locally analytic $G$-representation canonically carries the structure of a locally analytic $(\hy(G),H)$-module for any locally $L$-analytic subgroup $H$ of $G$.
\end{remark}

\begin{proposition}[{cf.\ \cite[Lemma 7.3.7, Lemma 7.4.2]{AgrawalStrauch22FromCatOLocAnRep}}]\label{Prop - Equivalent characterization for composite algebra action}
	Let $V$ be a separately continuous $D(H)$-modules which simultaneously is a separately continuous $\hy(G)$-module such that the two induced $\hy(H)$-actions agree.
	Then these two actions extend to a separately continuous $D(\dot{\Fg},H)$-module structure on $V$ if and only if 
	\begin{equation*}
		\delta_h \ast (\mu \ast v) = \Ad(h)(\mu) \ast (\delta_h \ast v) \quad\text{, for all $h \in H$, $\mu \in \hy(G)$, $v\in V$.}
	\end{equation*}
	In particular, any $(\hy(G),H)$-module naturally carries the structure of a separately continuous $D(\dot{\Fg},H)$-module.
\end{proposition}
\begin{proof}
	Given separately continuous $D(H)$- and $\hy(G)$-module structures on $V$ which agree on $\hy(H)$, the topological isomorphism of \Cref{Cor - Isomorphism for composite algebra} yields a well defined separately continuous homomorphism $D(\dot{\Fg},H) \times V \ra V $.
	This defines a separately continuous $D(\dot{\Fg},H)$-module structure on $V$ if and only if the scalar multiplication satisfies associativity.
	By the associativity of the $\hy(G)$- and $D(H)$-actions the latter is equivalent to
	\begin{equation*}
		(\delta_h \ast \mu) \ast v  = \delta_h \ast (\mu\ast v) \quad\text{, for all $h \in H$, $\mu \in \hy(G)$, $v\in V$.}
	\end{equation*}
	But like in the proof of \Cref{Cor - Compatibility with adjoint representation}, we find that $(\delta_h \ast \mu) \ast v = \Ad(h)(\mu) \ast (\delta_h \ast v)$.
\end{proof}

\begin{corollary}\label{Cor - Composite module from H-representation}
	For $W \in \CM_{D(H)}^\mathrm{nF}$, we set, for $\nu \otimes w \in \hy(G) \cot{\hy(H)} W$,
	\begin{align*}
		\mu \ast (\nu \otimes w) &\defeq \mu \ast \nu \otimes w && \text{, for $\mu \in \hy(G)$,} \\
		\delta \ast (\nu \otimes w) &\defeq \Ad(\delta)(\nu) \otimes \delta \ast w &&\text{, for $\delta \in D(H)$.}
	\end{align*}
	This defines a separately continuous $D(\dot{\Fg},H)$-module structure such that $\hy(G) \cot{\hy(H)} W $ is contained in $ \CM_{D(\dot{\Fg},H)}^\mathrm{nF}$.
\end{corollary}
\begin{proof}
	It is clear that the above defines a jointly continuous $\hy(G)$-module structure on $\hy(G) \ot{\hy(H)} W$.
	It extends to $\hy(G) \cot{\hy(H)} W$ by \cite[Lemma 1.2.2]{Emerton17LocAnVect}.

	Concerning the $D(H)$-action, we have seen in the proof of \Cref{Lemma - Module structure on tensor product} that the homomorphism $D(H) \ra D(H) \ot{K,\iota} D(H)$, $\delta \mto \delta \otimes \delta$, is continuous. 
	Since $\Ad$ is separately continuous, it follows that
	\begin{align*}
		 D(H) \ot{K,\iota} \hy(G) \ot{K,\iota} W \lra \hy(G) \ot{K,\iota} W \,,\quad
		 \delta \otimes \nu \otimes w \lto  \Ad(\delta)(\nu) \otimes \delta \ast w,
	\end{align*}
	is continuous as well.
	In this way $\delta \in D(H)$ induces a well-defined endomorphism of $\hy(G) \ot{\hy(H)} W$.
	Indeed, we compute that in the latter, for $\nu \in \hy(G)$, $\lambda \in \hy(H)$ and $w\in W$,
	\begin{align*}
		\Ad(\delta)(\nu \ast \lambda) \otimes \delta \ast w 
			= \Ad(\delta)(\nu) \otimes \Ad(\delta)(\lambda) \ast (\delta \ast w )
			= \Ad(\delta)(\nu) \otimes \delta \ast (\lambda \ast w )
	\end{align*}
	because $\Ad(\delta)$ is an algebra homomorphism and by applying \Cref{Cor - Compatibility with adjoint representation}.
	Moreover, using that $\hy(G)$ is a $D(H)$-module under $\Ad$ one also verifies that the above defines a $D(H)$-module structure on $\hy(G) \ot{\hy(H)} W$.
	\Cref{Lemma - Extending separately continuous module structure to completion} then shows that it extends to a separately continuous $D(H)$-module structure on $\hy(G) \cot{\hy(H)} W$.

	By \cite[Lemma 19.10 (i)]{Schneider02NonArchFunctAna}, we have $\hy(G) \cot{K} W \cong D_1(G) \cot{K} W$, which is a nuclear $K$-Fr\'echet space by the comment after \cite[Prop.\ 20.4]{Schneider02NonArchFunctAna}.
	Since by \Cref{Lemma - Completion of projective tensor product over algebras} $\hy(G) \cot{\hy(H)} W$ is a quotient of $\hy(G) \cot{K} W$, it follows that the former is a nuclear $K$-Fr\'echet space.

	To obtain $\hy(G) \cot{\hy(H)} W \in \CM_{D(\dot{\Fg},H)}^\mathrm{nF}$, it remains to prove 
	that the two induced $\hy(H)$-actions agree and that the compatibility condition of \Cref{Prop - Equivalent characterization for composite algebra action} is satisfied.
	For both it suffices to do so in $\hy(G) \ot{\hy(H)} W$, and one verifies the latter compatibility condition there directly.
	On the other hand let again $\lambda \in \hy(H)$, $\nu \in \hy(G)$ and $w\in W$.
	Using \Cref{Lemma - Formula for adjoint representation} we have in $\hy(G) \ot{\hy(H)} W$:
	\begin{align*}
		\Ad(\lambda)(\nu) \otimes \lambda \ast w = \Ad(\lambda)(\nu) \ast \lambda \otimes  w 
			= \lambda \ast \nu \otimes  w,
	\end{align*}
	which shows that the two $\hy(H)$-actions agree.
\end{proof}

\subsection{Definition and Left Exactness of the Functors $\dot{\CF}^G_P$}

Let $M \in  \CM_{D(\dot{\Fg},H)}$ and consider the inclusion $D(\dot{\Fg},H) \hookrightarrow D(G)$ of separately continuous $K$-algebras.
By \Cref{Lemma - Base change for topological modules} (i)  the induced $D(G)$-module $D(G) \ot{D(\dot{\Fg},H),\iota} M$ is separately continuous.

\begin{lemma}\label{Lemma - Tensoring up from composite algebra}
	For $M \in \CM_{D(\dot{\Fg},H)}$, the embedding $D(G_0) \hookrightarrow D(G)$ induces a topological isomorphism of separately continuous $D(G_0)$-modules
	\begin{equation*}
		D(G_0) \ot{D(\dot{\Fg},H_0),\iota} M \overset{\cong}{\lra} D(G) \ot{D(\dot{\Fg},H),\iota} M .
	\end{equation*}
\end{lemma}
\begin{proof}
	We can prolong the map \eqref{Eq - Homomorphism for tensoring up} from the proof of \Cref{Lemma - Tensoring up from distribution algebras of subgroups} (ii) to obtain a continuous homomorphism
	\begin{equation*}
		D(G) \ot{K,\iota} M \lra D(G_0) \ot{D(H_0),\iota} M \longtwoheadrightarrow D(G_0) \ot{D(\dot{\Fg},H_0),\iota} M . 
	\end{equation*}
	It factors over the quotient $D(G) \ot{D(\dot{\Fg},H),\iota} M$ to give a continuous inverse to the claimed isomorphism.
\end{proof}

\begin{definition}\label{Def - Functor by tensoring up}
	For $M \in \CM_{D(\dot{\Fg},H)}$ we define
	\begin{equation*}
		\dot{\CF}^G_H(M) \defeq \big( D(G) \cot{D(\dot{\Fg},H),\iota} M \big)'_b . 
	\end{equation*}
\end{definition}

Note that in the situation of $G=\bG(L)$, $H = \bP(L)$, for a connected reductive algebraic group $\bG$ over $L$ with a parabolic subgroup $\bP$, we recover the definition of $\dot{\CF}_P^G(M) $ from the introduction.

\begin{proposition}\label{Prop - Locally analytic representation via tensoring up}
	If $M \in  \CM_{D(\dot{\Fg},H)}$ is pseudo-metrizable and nuclear, then $\dot{\CF}_H^G(M)$ has a canonical structure of a locally analytic $G$-representation of compact type.
\end{proposition}
\begin{proof}
	As we have seen in \Cref{Lemma - Tensoring up from composite algebra}, $D(G) \ot{D(\dot{\Fg},H),\iota} M$ is topologically isomorphic to $D(G_0) \ot{D(\dot{\Fg},H_0)} M $ as a locally convex $K$-vector space.
	Since $D(G_0)$ is a nuclear $K$-Fr\'echet space, \Cref{Lemma - Properties of pseudo-metrizable spaces} (i) and \cite[Prop.\ 19.11]{Schneider02NonArchFunctAna} thus imply that $D(G) \ot{D(\dot{\Fg},H),\iota} M$ is pseudo-metrizable and nuclear.

	Because $D(G)$ is barrelled (\Cref{Prop - Properties of space of distributions} (i)), it follows from \Cref{Lemma - Extending separately continuous module structure to completion} that the $D(G)$-module structure on $D(G) \ot{D(\dot{\Fg},H),\iota} M$ extends to a separately continuous $D(G)$-module structure on the Hausdorff completion $D(G) \cot{D(\dot{\Fg},H),\iota} M$.
	By \cite[Prop.\ 20.4]{Schneider02NonArchFunctAna} this completion is nuclear again, so that $D(G) \cot{D(\dot{\Fg},H),\iota} M \in \CM_{D(G)}^\mathrm{nF}$.
	Via the anti-equivalence of \Cref{Prop - Anti-equivalence for locally analytic representations} (ii), the strong dual space $\dot{\CF}^G_H(M)$ is a locally analytic $G$-representation of compact type.
\end{proof}

\begin{remark}
	For $V\in \Rep_K^{\la,\mathrm{ct}}(H)$, we have seen in \Cref{Cor - Composite module from H-representation} that $\hy(G) \cot{\hy(H)} V'_b$ is contained in $ \CM_{D(\dot{\Fg},H)}^\mathrm{nF}$.
	\Cref{Lemma - Tensor identities for modules} and \Cref{Cor - Isomorphism for composite algebra} then yield a topological isomorphism of separately continuous $D(G)$-modules
	\begin{align*}
		D(G) \ot{D(\dot{\Fg},H),\iota} \big( \hy(G) \ot{\hy(H),\iota} V'_b \big) 
			&\cong D(G) \ot{D(\dot{\Fg},H),\iota} \Big( \big( \hy(G) \ot{\hy(H),\iota} D(H) \big) \ot{D(H),\iota} V'_b \Big) \\
			&\cong D(G) \ot{D(H),\iota} V'_b .
	\end{align*}
	Combined with \Cref{Prop - Module description for induction} we obtain an isomorphism of locally analytic $G$-representations
	\begin{equation*}
		\dot{\CF}_H^G \big( \hy(G) \cot{\hy(H)} V'_b  \big) \cong \Ind_{H}^G \big( V \big) .
	\end{equation*}
\end{remark}

\begin{theorem}\label{Thm - Functor via tensoring up is left exact}
	\Cref{Def - Functor by tensoring up} yields a left exact\footnote{In the sense of \cite[1.1.5]{Schneiders98QuasiAbCat}, i.e. for a short strictly exact sequence $0 \ra M_1 \xrightarrow{f} M_2 \xrightarrow{g} M_3 \ra 0 $ in $\CM_{D(\dot{\Fg},H)}^\mathrm{nF}$, the induced $\dot{\CF}_H^G(g)$ is a strict monomorphism in $\Rep_K^{\la,\mathrm{ct}}(G)$ and $\Im\big(\dot{\CF}_H^G(g)\big) = \Ker\big(\dot{\CF}_H^G(f) \big)$.} contravariant functor
	\begin{equation*}
		\dot{\CF}^G_H \colon \CM_{D(\dot{\Fg},H)}^\mathrm{nF} \lra \Rep_K^{\la, \mathrm{ct}} (G) \,,\quad M \lto \dot{\CF}^G_H(M) , 
	\end{equation*}
	between quasi-abelian categories.
\end{theorem}
\begin{proof}
	The functoriality of $\dot{\CF}^G_H\colon \CM_{D(\dot{\Fg},H)}^\mathrm{nF} \ra \Rep_K^{\la, \mathrm{ct}} (G)$ is clear.
	It remains to show that $\dot{\CF}_H^G$ is left exact.
	To this end, let
	\begin{equation*}
		0 \lra M_1 \overset{f}{\lra} M_2 \overset{g}{\lra} M_3 \lra 0 
	\end{equation*}
	be a short strictly exact sequence in $\CM_{D(\dot{\Fg},H)}^\mathrm{nF}$.
	Under the anti-equivalence of \Cref{Prop - Anti-equivalence for locally analytic representations} (ii) and using \Cref{Lemma - Tensoring up from composite algebra}, it suffices to show that 
	\begin{equation}\label{Eq - Tensored up sequence}
		D(G_0) \cot{D(\dot{\Fg},H_0)} M_1 \overset{\widehat{f}}{\lra } D(G_0) \cot{D(\dot{\Fg},H_0)} M_2  \overset{\widehat{g}}{\lra } D(G_0) \cot{D(\dot{\Fg},H_0)} M_3  \lra 0 
	\end{equation}
	is strictly coexact, i.e.\ that $\ov{\Im(\widehat{f})}=\Ker(\widehat{g})$ and that $\widehat{g}$ is a strict epimorphism.

	Concerning the latter, by \Cref{Lemma - Completion of projective tensor product over algebras} we have a commutative square
	\begin{equation*}
		\begin{tikzcd}
				D(G_0) \cot{K} M_2 \ar[r, "\id \cot{} g"] \ar[d, two heads] & D(G_0) \cot{K} M_3 \ar[d, two heads] \\
			D(G_0) \cot{D(\dot{\Fg},H_0)} M_2  \ar[r, "\widehat{g}"] &D(G_0) \cot{D(\dot{\Fg},H_0)} M_3
		\end{tikzcd}
	\end{equation*}
	where the vertical maps are strict epimorphisms. 
	Moreover, $\id \cot{} g$ is a strict epimorphism by \cite[Cor.\ 2.2]{BreuilHerzig18TowardsFinSlopePartGLn}.
	It follows from \cite[Lemma 2]{KopylovKuzminov00KerCokerSeqSemiAbCat} that $\widehat{g}$ is a strict epimorphism as well.

	Functoriality implies that $\ov{\Im(\widehat{f})} \subset \Ker(\widehat{g})$.
	Consider the commutative diagram
		\begin{equation}
		\begin{tikzcd}
			0 \ar[r] &D(G_0) \ot{K} M_1 \ar[r, "\id \ot{} f"] \ar[d, two heads] & D(G_0) \ot{K} M_2 \ar[r, "\id \ot{} g"] \ar[d, two heads] & D(G_0) \ot{K} M_3 \ar[r] \ar[d, two heads] & 0 \\
			&D(G_0) \ot{D(\dot{\Fg},H_0)} M_1 \ar[r, "\ov{f}"] &D(G_0) \ot{D(\dot{\Fg},H_0)} M_2  \ar[r, "\ov{g}"] &D(G_0) \ot{D(\dot{\Fg},H_0)} M_3 \ar[r] &0
		\end{tikzcd}
	\end{equation}
	whose bottom row yields \eqref{Eq - Tensored up sequence} via taking the Hausdorff completion.
	In particular, since $D(G_0) \cot{D(\dot{\Fg},H_0)} M_2$ is hereditarily complete, it follows from \cite[Cor.\ 2.2]{BreuilHerzig18TowardsFinSlopePartGLn} that $\Ker(\widehat{g})$ is the completion of $\Ker(\ov{g})$.
	Therefore it suffices to show that $ \Ker(\ov{g}) \subset \Im(\ov{f}) $.

	For $i=2,3$, let $M'_i$ denote the subspace of $D(G_0) \ot{K} M_i$ generated by 
	\begin{equation*}
		\delta \ast \lambda \otimes m - \delta \otimes \lambda \ast m \quad\text{, for $\delta \in D(G_0)$, $\lambda \in D(\dot{\Fg},H_0)$, $m\in M_i$.}
	\end{equation*}
	Since $g$ is surjective, for every element of $M'_3$, we find a preimage under $\id \ot{} g$ which is contained in $M'_2$.
	Using that $\Im(\id \ot{} f) = \Ker(\id \ot{} g)$, a diagram chase then yields $ \Ker(\ov{g}) \subset \Im(\ov{f}) $.
\end{proof}

\begin{remark}
	For $\mathrm{char}(L)=0$, the Orlik--Strauch functors $\CF_P^G$ are exact \cite[Prop.\ 4.9]{OrlikStrauch15JordanHoelderSerLocAnRep}.
	In view of the compatibility between $\CF_P^G$ and the above $\dot{\CF}_P^G$ (see \Cref{Sect - Dual of Orlik-Strauch functors}) we expect $\dot{\CF}^G_P$ to be an exact functor between quasi-abelian categories for $L$ of arbitrary characteristic.
	As additional evidence, recall that for $M \in \CM_{D(\dot{\Fg},H)}^\mathrm{nF}$ of the form $M = \hy(G) \cot{\hy(H)} W$, for $W\in \CM_{D(H)}^\mathrm{nF}$, we have $\dot{\CF}_H^G(M)\cong \Ind_H^{\la,G}(W'_b)$ and this locally analytic induction functor is exact (see \Cref{Cor - Locally analytic induction is exact}).
\end{remark}

\section{Comparison with the $p$-adic Orlik--Strauch Functors}\label{Sect - Comparison}

In this section we show that the functors $\dot{\CF}^G_P$ from \Cref{Def - Functor by tensoring up} generalize the Orlik--Strauch functors $\CF_P^G$ when $L$ is $p$-adic.
As in the introduction, let $\bG$ be a connected split reductive group over $L$, and let $\bT \subset \bB$ be a split maximal torus and a Borel subgroup of $\bG$, respectively.
Let $\bP$ be a standard parabolic subgroup of $\bG$ with Levi decomposition $\bP = \bL_\bP \bU_\bP$ where $\bT \subset \bL_\bP$.
We let $G=\bG(L)$, $P=\bP(L)$, etc.\ denote the associated locally $L$-analytic Lie groups.
Moreover, let $G_0 \subset G$ be a maximal compact open subgroup (cf.\ the beginning of \Cref{Sect - The functors FGP}), and $P_0 \defeq G_0 \cap P$.

Finally, let $\Fg$, $\Fp$, etc.\ denote the corresponding Lie algebras.
We use the convention to write $U(\Fg)$ in place of $U(\Fg_K)$ and analogously for the other universal enveloping algebras.
In view of \Cref{Prop - Hyperalgebra and universal enveloping algebra} we make the identifications $\hy(G) = U(\Fg)$ and $\hy(P) = U(\Fp)$.
Consequently the subalgebra $D(\dot{\Fg},P)$ of $D(G)$ agrees with the subalgebra $D(\Fg,P)$ defined in \cite[Sect.\ 3.4]{OrlikStrauch15JordanHoelderSerLocAnRep}.

Note that the $K$-Fr\'echet algebra $D(G_0)$ is a Fr\'echet--Stein algebra.
Thus one may consider coadmissible $D(G)$-modules or equivalently the abelian category $\Rep_K^{\la, \adm}(G)$ of admissible locally analytic $G$-representations, see \cite[Sect.\ 5, 6]{SchneiderTeitelbaum03AlgpAdicDistAdmRep}.

\subsection{The Category $\CO_\alg^\Fp$}
The following variants of the Bernstein--Gelfand--Gelfand category $\CO$ of $U(\Fg)$-modules play a key role.

\begin{definition}[{\cite[Sect.\  2.5]{OrlikStrauch15JordanHoelderSerLocAnRep}}]
	\begin{altenumerate}
		\item
		Let $\CO^\Fp$ be the full subcategory of $U(\Fg)$-modules $M$ which satisfy
		\begin{altenumeratelevel2}
			\item 
			$M$ is finitely generated as a $U(\Fg)$-module,
			\item
			viewed as an $\Fl_{\bP,K}$-module, $M$ is the direct sum of finite-dimensional simple modules,
			\item
			the action of $\Fu_{\bP,K}$ on $M$ is locally finite, i.e.\ for every $m\in M$, the $K$-vector subspace $U(\Fu_\bP) m \subset M$ is finite-dimensional.
		\end{altenumeratelevel2}
		\item
		Let ${\rm Irr}(\Fl_{\bP,K})^{\rm fd}$ be the set of isomorphism classes of finite-dimensional irreducible $\Fl_{\bP,K}$-re\-pre\-sen\-tations.
		We define $\CO_\alg^\Fp$ to be the full subcategory of $\CO^\Fp$ of $U(\Fg)$-modules $M$ such that for a decomposition 
		\[ M = \bigoplus_{\Fa \in {\rm Irr}(\Fl_{\bP,K})^{\rm fd}} M_\Fa \]
		into the $\Fa$-isotypic components as in (2) of (i), we have: If $M_\Fa \neq (0)$ then $\Fa$ is the Lie algebra representation induced by some finite-dimensional algebraic $\bL_{\bP,K}$-representation.
	\end{altenumerate}
\end{definition}

Every $M \in \CO_\alg^\Fp$ is the union of finite-dimensional $U(\Fp)$-submodules.
Hence one obtains an (abstract) $D(P)$-module structure on $M$ via lifting each of those to an algebraic $\bP_K$-representation, see \cite[\S 3.4]{OrlikStrauch15JordanHoelderSerLocAnRep}.
Moreover, Orlik and Strauch show:

\begin{lemma}[{\cite[Cor.\ 3.6]{OrlikStrauch15JordanHoelderSerLocAnRep}, cf.\ \cite[Prop.\ 3.3.2]{AgrawalStrauch22FromCatOLocAnRep}}]\label{Lemma - Module structure on object in category O}
	On any $M \in \CO^\Fp_\alg$, there exists a unique (abstract) $D(\Fg, P)$-module structure which extends the given $U(\Fg)$-module structure and such that the Dirac distributions $\delta_p \in D(P)$ act like the group elements $p\in P$ under the above $P$-action on $M$.
	In particular
	\begin{equation*}
		\Ad(p)(\mu) \ast (p.m) = p.(\mu \ast m) \quad\text{, for all $p\in P$, $\mu \in U(\Fg)$, $m \in M$.}
	\end{equation*} 
\end{lemma}

In the following we consider $U(\Fg)$ and $U(\Fp)$ as jointly continuous subalgebras of $D(G)$.
As every $M \in \CO_\alg^\Fp$ is finitely generated over $U(\Fg)$, we may endow $M$ with the quotient topology induced by some epimorphism $U(\Fg)^{\oplus n} \twoheadrightarrow M$ of $U(\Fg)$-modules, cf.\ \Cref{Prop - Topologies on finitely generated modules}.

\begin{proposition}\label{Prop - Composite module structure on object in category O}
	For any $M\in \CO_\alg^\Fp$ endowed with the quotient topology, the $D(\Fg,P)$-module structure from \Cref{Lemma - Module structure on object in category O} is separately continuous.
\end{proposition}
\begin{proof}
	By \Cref{Prop - Equivalent characterization for composite algebra action} and \Cref{Lemma - Module structure on object in category O} it suffices to show that the $D(P)$-action on $M$ is a separately continuous with respect to the quotient topology.
	Since $M$ is the union of finite-dimensional locally analytic $P$-representations, for each $m\in M$, the homomorphism $D(P)\ra M$, $\delta \mto \delta \ast m$, is continuous.

	We fix $\delta \in D(P)$ and want to show that $f_\delta \colon M \ra M$, $m \mto \delta \ast m$, is continuous.
	Let $\varphi\colon U(\Fg)^{\oplus n} \twoheadrightarrow M$ be a strict epimorphism endowing $M$ with its topology.
	Let $e_1,\ldots,e_n $ denote the standard basis elements of $U(\Fg)^{\oplus n}$.
	Then by \Cref{Lemma - Module structure on object in category O} we have for the composite $\varphi'\defeq f_\delta \circ \varphi$:
	\begin{align*}
		\varphi'(\mu_1,\ldots,\mu_n) = \sum_{i=1}^n \delta \ast \mu_i \ast \varphi(e_i) = \sum_{i=1}^n \Ad(\delta)(\mu_i) \ast \big(\delta \ast \varphi(e_i)\big) , 
	\end{align*}
	for all $\mu_1,\ldots,\mu_n \in U(\Fg)$.
	Because $\Ad(\delta)$ is continuous (see \Cref{Prop - Adjoint representation}), this shows that $\varphi'$ is continuous.
	Since $\varphi$ is open, it follows that $f_\delta$ is continuous.
\end{proof}

\begin{corollary}\label{Cor - Composite module structure on completion of object in category O}
	For $M \in \CO_\alg^\Fp$ the above separately continuous $D(\Fg,P)$-module structure extends to a separately continuous $D(\Fg,P)$-module structure on $\widehat{M}$.
	This yields a strongly exact functor
	\begin{equation*}
		\CC \colon \CO_\alg^\Fp \lra \CM_{D(\dot{\Fg},P)}^\mathrm{nF} \,,\quad M \lto \widehat{M} ,
	\end{equation*}
	i.e.\ it sends short exact sequences in $ \CO_\alg^\Fp$ to short strictly exact sequences in $\CM_{D(\dot{\Fg},P)}^\mathrm{nF} $.
\end{corollary}
\begin{proof}
	Since $U(\Fg)$ is pseudo-metrizable and nuclear as the subspace of the nuclear $K$-Fr\'echet space $D(G_0)$, also $M$ with the quotient topology is pseudo-metrizable and nuclear, see \Cref{Lemma - Properties of pseudo-metrizable spaces} and \cite[Prop.\ 19.4]{Schneider02NonArchFunctAna}.
	Thus $\widehat{M}$ is a nuclear $K$-Fr\'echet space by \cite[Prop.\ 20.4]{Schneider02NonArchFunctAna}.

	It follows from \Cref{Lemma - Extending separately continuous module structure to completion} that the induced separately continuous $D(P)$-module structure on $M$ extends to a separately continuous $D(P)$-module structure on $\widehat{M}$.
	Moreover, the jointly continuous $U(\Fg)$-module structure on $M$ extends to one on $\widehat{M}$ by \cite[Lemma 1.2.2]{Emerton17LocAnVect}.
	These two extended actions still satisfy the compatibility conditions of \Cref{Prop - Equivalent characterization for composite algebra action} because of continuity in $\widehat{M}$.
	Therefore the former proposition shows that $\widehat{M} \in \CM_{D(\dot{\Fg},P)}^\mathrm{nF}$.

	\Cref{Prop - Topologies on finitely generated modules} shows that every short exact sequence in $\CO_\alg^\Fp$ is a short strictly exact sequence in $\mathrm{LCS}_K$ when the occurring $U(\Fg)$-modules carry the quotient topology.
	It follows from \cite[Cor.\ 2.2]{BreuilHerzig18TowardsFinSlopePartGLn} that the completion of such a short strictly exact sequence is strictly exact again.
\end{proof}

\subsection{The Dual of the Orlik--Strauch Functors $\CF_P^G$}\label{Sect - Dual of Orlik-Strauch functors}

We first recapitulate the construction of the functors $\CF_P^G$ from \cite{OrlikStrauch15JordanHoelderSerLocAnRep}.

Let $V$ be an admissible smooth $P$-representation\footnote{In \cite{OrlikStrauch15JordanHoelderSerLocAnRep} only those representations arising via inflation from an admissible smooth representation of $L_\bP$ are considered, but the following construction works in the generality as stated.} on a $K$-vector space.
It becomes an admissible locally analytic $P$-representation when endowed with the finest locally convex topology \cite[Thm.\ 6.6 (i)]{SchneiderTeitelbaum03AlgpAdicDistAdmRep} and is of compact type.
We will consider such admissible smooth representations always in this way.

For any $M \in \CO^\Fp_\alg$, there exists a finite-dimensional $U(\Fp)$-submodule $W\subset M$ which generates $M$ as a $U(\Fg)$-module, i.e.\ there is a short exact sequence of $U(\Fg)$-modules
\begin{equation*}
	0 \lra \Fd \lra U(\Fg) \ot{U(\Fp)} W \lra M \lra 0  .
\end{equation*}
Then the $\Fp_K$-re\-pre\-sen\-ta\-tion $W$ uniquely lifts to the structure of an algebraic $\bP_K$-re\-pre\-sen\-ta\-tion on $W$ \cite[Lemma 3.2]{OrlikStrauch15JordanHoelderSerLocAnRep}.

Recall the identification $W' \ot{K} V \cong \CL_b (W,V)$ from \cite[Cor.\ 18.8]{Schneider02NonArchFunctAna} and let 
\begin{equation*}
	\ev_w \colon \CL_b(W,V) \lra V\,,\quad h \lto h(w),
\end{equation*}
denote the evaluation homomorphism.
Then there exists a pairing (cf.\ \cite[(3.2.2)]{OrlikStrauch15JordanHoelderSerLocAnRep})
\begin{align*}
	\langle \blank , \!\blank \rangle_{C^\la(G,V)} \colon 
	D(G) \ot{K} W  \times C^\la \big(G, W'\ot{K} V \big) &\lra C^\la(G,V) , \\
	(\delta \otimes w , f) &\lto \Big[ g \mto  \delta \big[ x \mto \ev_w(f(gx))\big] \Big]	.
\end{align*}
Here $\delta$ is applied to the function $\big[x\mto \ev_w(f(gx))\big] \in C^\la(G,V)$ via the pairing from \Cref{Cor - Pairing of distributions and vector valued functions}.
We can embed $U(\Fg) \ot{U(\Fp)} W \hookrightarrow D(G) \ot{D(P)} W$ and, for $\Fz \in U(\Fg) \ot{U(\Fp)} W$ and $f\in \Ind^{\la,G}_P ( W'\ot{K} V)$, one finds that $\langle \Fz , f \rangle_{C^\la(G,V)}$ is well-defined, cf.\ \cite[Lemme 2.1]{Breuil16SocLocAnalI}.
This allows the definition \cite[(4.4.1)]{OrlikStrauch15JordanHoelderSerLocAnRep}
\begin{equation*}
	\CF^G_P (M, V) \defeq \Ind^{\la,G}_P \big(W'\ot{K} V\big)^\Fd = \Big\{ f\in \Ind^{\la,G}_P \big( W'\ot{K} V\big) \,\Big\vert\, \forall \Fz \in \Fd : \langle \Fz, f \rangle_{C^\la(G,V)} = 0 \Big\} .
\end{equation*}

The resulting $\CF^G_P(M,V)$ is an admissible\footnote{In the sense of \cite[\S 6]{SchneiderTeitelbaum03AlgpAdicDistAdmRep}.} locally analytic $G$-representation which does not depend on the choice of $W$ \cite[Prop.\ 4.5]{OrlikStrauch15JordanHoelderSerLocAnRep}.
%It even is strongly admissible\footnote{Meaning that as a representation of any (equivalently, of one) compact open subgroup $H \subset G$, it is strongly admissible in the sense of \cite[\S 3]{SchneiderTeitelbaum02LocAnDistApplToGL2}, i.e.\ its strong dual is finitely generated as a $D(H)$-module.} if $V$ is of finite length \cite[Prop.\ 4.8]{OrlikStrauch15JordanHoelderSerLocAnRep}.
This construction yields an exact bi-functor 
\begin{equation*}
	\CF^G_P \colon \CO^\Fp_\alg \times {\rm Rep}_K^{\sm, {\rm adm}}(L_\bP) \lra {\rm Rep}_K^{\la,{\rm adm}}(G) \,,\quad (M,V) \lto \CF_P^G(M,V) ,
\end{equation*}
which is contravariant in $M$ and covariant in $V$, see \cite[Prop.\ 4.7]{OrlikStrauch15JordanHoelderSerLocAnRep}.

\medskip

We want to explore another description of these functors in terms of $D(G)$-modules.
For the smooth $P$-representations we restrict ourselves to those which are strongly admissible in the sense of \cite[Sect.\ 2 Def.]{SchneiderTeitelbaumPrasad01UgFinLocAnRep}.
We recall an equivalent definition, cf.\ \cite[Prop.\ 2.1]{SchneiderTeitelbaumPrasad01UgFinLocAnRep}, \cite[Sect.\ 4.1.2]{AgrawalStrauch22FromCatOLocAnRep}.

\begin{definition}
	A smooth representation $V$ of a locally $L$-analytic Lie group $H$ on a $K$-vector space is called \textit{strongly admissible} if $V$ embeds as a subrepresentation into $C^\sm (H_0,K)^{\oplus n}$, for some compact open subgroup $H_0 \subset H$ and some $n\in \BN$, when viewed as a representation of $H_0$.
	Note that smooth representations of finite length are strongly admissible \cite[Sect.\ 4.1]{AgrawalStrauch22FromCatOLocAnRep}, and strongly admissible smooth representations are admissible.
	We denote the category of strongly admissible $H$-representations on $K$-vector spaces by $\Rep_K^{\sm,\mathrm{s\text{-}adm}}(H)$.
\end{definition}

For $M\in \CO^\Fp_\alg$ and a strongly admissible smooth $P$-representation $V$ it is a special case of a result of Agrawal and Strauch \cite[Thm.\ 4.2.3]{AgrawalStrauch22FromCatOLocAnRep} that the (abstract) $D(G)$-module
\begin{equation*}\label{Eq - Module over D(g,P) tensored up}
	D(G) \ot{D(\Fg,P)} \big(M \ot{K} V' \big)
\end{equation*}
is coadmissible.
In particular, this $D(G)$-module carries a canonical Fr\'echet topology \cite[Sect.\  3]{SchneiderTeitelbaum03AlgpAdicDistAdmRep}.

On the other hand, we have seen in \Cref{Prop - Composite module structure on object in category O} that $M$ with its quotient topology is a separately continuous $D(\Fg,P)$-module.
Moreover, $V'_b$ is a separately continuous $D(P)$-module with trivial $U(\Fp)$-action because the representation $V$ is smooth.
It follows from \Cref{Prop - Equivalent characterization for composite algebra action} that $V'_b$ endowed with the trivial $U(\Fg)$-action is a separately continuous $D(\Fg,P)$-module.
Then the projective and inductive tensor product topologies agree on $M \ot{K} V'_b$ by \Cref{Lemma - Equality of projective and inductive tensor product} and $M \ot{K} V'_b \in \CM_{D(\Fg,P)}$ by \Cref{Lemma - Module structure on tensor product}.
Thus the $D(G)$-module structure on $D(G) \ot{D(\Fg,P),\iota} \big(M \ot{K} V'_b \big)$ is separately continuous by \Cref{Lemma - Base change for topological modules} (i).

\begin{theorem}\label{Thm - Comparision between canonical Frechet topology and topological tensor product}
	Let $M \in \CO^\Fp_\alg$ be endowed with the quotient topology, and let $V$ be a strongly admissible smooth $P$-representation.
	Then the topology on the separately continuous $D(G)$-module
	\begin{equation*}
		D(G) \ot{D(\Fg,P),\iota} \big(M \ot{K} V'_b \big)
	\end{equation*}
	agrees with the canonical Fr\'echet topology induced from its underlying abstract $D(G)$-module being coadmissible.
\end{theorem}
\begin{proof}
	By the assumptions on $M \in \CO^\Fp_\alg$, we may find a finite-dimensional $U(\Fp)$-module $W$ and an epimorphism $\varphi \colon U(\Fg) \ot{U(\Fp)} W \twoheadrightarrow M$ of $U(\Fg)$-modules.
	The composite
	\begin{equation*}
		U(\Fg)^{\oplus \dim_K(W)} \cong U(\Fg) \ot{K} W \longtwoheadrightarrow U(\Fg) \ot{U(\Fp)} W \overset{\varphi}{\longtwoheadrightarrow} M
	\end{equation*}
	then induces the quotient topology on $M$.
	It follows from \cite[Lemma 2]{KopylovKuzminov00KerCokerSeqSemiAbCat} that $\varphi$ is a strict epimorphism with respect to this choice of topology on $M$.
	Using \Cref{Cor - Isomorphism for composite algebra} and \Cref{Lemma - Tensor identities for modules} there is a topological isomorphism
	\begin{equation*}
		D(\Fg,P_0) \ot{D(P_0)} W 
		\cong \big( U(\Fg) \ot{U(\Fp)} D(P_0) \big) \ot{D(P_0)} W  
		\cong U(\Fg) \ot{U(\Fp)} W ,
	\end{equation*}
	which maps $\mu \ast \delta \otimes w$ to $\mu \otimes \delta \ast w$, for $\mu \in U(\Fg)$, $\delta\in D(P_0)$, $w\in W$.
	Therefore $\varphi$ induces a strict epimorphism that we continue to denote by 
	\begin{equation*}
	 	\varphi \colon	D(\Fg,P_0) \ot{D(P_0)} W  \longtwoheadrightarrow M .
	\end{equation*}
	Expressing $\delta \ast \mu$, for $\delta \in D(P_0)$, $\mu \in U(\Fg)$, as $\delta \ast \mu = \sum_{i=1}^r \mu_i \ast \delta_i$, for $\mu_i \in U(\Fg)$, $\delta_i \in D(P_0)$, as in the proof of \Cref{Prop - Description of composite subalgebra}, one verifies that $\varphi$ is $D(\Fg,P_0)$-linear.

	Since the projective tensor product is exact (see \cite[Lemma 2.1 (ii)]{BreuilHerzig18TowardsFinSlopePartGLn}), we obtain a strict epimorphism
	\begin{equation}\label{Eq - Strict epimorphism of composite algebra modules}
		D(\Fg,P_0) \ot{D(P_0)} \big(  W \ot{K} V'_b\big)  \cong \big( D(\Fg,P_0) \ot{D(P_0)} W \big) \ot{K} V'_b \xtwoheadrightarrow{\!\!\varphi \ot{} \id \!\!} M \ot{K} V'_b .
	\end{equation}
	We extend the trivial $U(\Fp)$-action on $V'_b$ to the trivial $U(\Fg)$-action.
	Then \eqref{Eq - Strict epimorphism of composite algebra modules} is $D(\Fg,P_0)$-linear with respect to $D(\Fg,P_0)$ acting on the first factor of the source and the diagonal action on the target. 
	Moreover, the jointly continuous $D(P_0)$-module $W \ot{K} V'_b$ is finitely generated as an abstract $D(P_0)$-module, see \cite[Prop.\ 6.4.1]{AgrawalStrauch22FromCatOLocAnRep}.
	Such an epimorphism $D(P_0)^{\oplus n} \twoheadrightarrow W \ot{K} V'_b$, for some $n\in \BN$, is continuous and necessarily strict by the open mapping theorem \cite[Prop.\ 8.6]{Schneider02NonArchFunctAna} since $W\ot{K} V'_b$ is a $K$-Fr\'echet space.
	Therefore we obtain a commutative diagram 
	\begin{equation*}
		\begin{tikzcd}
			D(\Fg,P_0) \ot{K} D(P_0)^{\oplus n} \ar[r, two heads] \ar[d, two heads] & D(\Fg,P_0) \ot{K} \big(W \ot{K} V'_b \big) \ar[d, two heads] \\
			D(\Fg,P_0)^{\oplus n} \ar[r] & D(\Fg,P_0) \ot{D(P_0)} \big(W\ot{K} V'_b \big)
		\end{tikzcd}
	\end{equation*}
	where the top homomorphism is a strict epimorphism by \cite[Lemma 2.1 (ii)]{BreuilHerzig18TowardsFinSlopePartGLn}.
	It follows from \cite[Lemma 2]{KopylovKuzminov00KerCokerSeqSemiAbCat} that the bottom epimorphism is strict as well.
	Taking the composition of this bottom epimorphism and $\varphi$ we thus arrive at a strict epimorphism of jointly continuous $D(\Fg,P_0)$-modules
	\begin{equation*}
		\psi \colon D(\Fg,P_0)^{\oplus n} \longtwoheadrightarrow M \ot{K} V'_b .
	\end{equation*}

	Via the exactness of the projective tensor product, we obtain the commutative diagram
	\begin{equation*}
		\begin{tikzcd}
			D(G_0) \ot{K} D(\Fg,P_0)^{\oplus n} \ar[d, two heads] \ar[r, two heads, "\id \ot{} \psi"] & D(G_0) \ot{K} \big(M \ot{K} V'_b \big) \ar[d, two heads]  \\
			D(G_0)^{\oplus n} \ar[r, "\ov{\psi}"] & D(G_0) \ot{D(\Fg,P_0)} \big(M \ot{K} V'_b \big) 
		\end{tikzcd}
	\end{equation*}
	where all maps are strict epimorphisms -- for $\ov{\psi}$ one deduces this as before.
	On the other hand, the (abstract) $D(G_0)$-module $D(G_0) \ot{D(\Fg,P_0)} \big(M \ot{K} V'_b \big)$ is coadmissible by \cite[Thm.\ 4.2.3]{AgrawalStrauch22FromCatOLocAnRep}.
	Therefore $\ov{\psi}$ also is strict when $D(G_0) \ot{D(\Fg,P_0)} \big(M \ot{K} V'_b \big)$ carries its canonical Fr\'echet topology (see \cite[Sect.\ 3]{SchneiderTeitelbaum03AlgpAdicDistAdmRep}).
	By the uniqueness of the quotient we conclude that this Fr\'echet topology agrees with the topology on $D(G_0) \ot{D(\Fg,P_0)} \big(M \ot{K} V'_b \big) $.
	The theorem now follows from the topological isomorphism of $D(G_0)$-modules from \Cref{Lemma - Tensoring up from composite algebra}
	\begin{equation*}
		D(G) \ot{D(\Fg,P),\iota} \big(M \ot{K} V' \big) \cong D(G_0) \ot{D(\Fg,P_0)} \big(M \ot{K} V' \big) .
	\end{equation*}
\end{proof}

\begin{lemma}\label{Lemma - Some duality pairing}
	Let $W$ be a finite-dimensional algebraic $\bP_K$-representation and $V$ a strongly admissible smooth $P$-representation.
	Then there exists a $G$-equivariant duality pairing
	\begin{align*}
		\langle \blank , \!\blank \rangle \colon 
		D(G) \ot{D(\Fg,P),\iota} \Big( \big( U(\Fg) \ot{U(\Fp)} W \big) \ot{K} V'_b \Big)  \times \Ind^{\la, G}_P \big( W'\ot{K} V \big) &\lra K , \\
		(\delta \otimes \mu \otimes w \otimes \ell , f) &\lto (w\otimes \ell) \big( (\delta \ast \mu)(f) \big)	,
	\end{align*}
	where $\delta \ast \mu$ is applied to $f\in C^\la\big(G,W'\ot{K} V \big)$ via the pairing of \Cref{Cor - Pairing of distributions and vector valued functions}.
\end{lemma}
\begin{proof}
	We have $U(\Fg) \ot{U(\Fp)} W \in \CO_\alg^\Fp$, and with the induced separately continuous $D(\Fg,P)$-module structure \Cref{Cor - Isomorphism for composite algebra} and \Cref{Lemma - Tensor identities for modules} yield a topological isomorphism of $D(G)$-modules
	\begin{align}\label{Eq - Generalized Verma module over D(g,P) tensored up}
		D(G) \ot{D(\Fg,P),\iota} \Big( \big( U(\Fg) \ot{U(\Fp)} W \big) \ot{K} V'_b \Big) &\overset{\cong}{\lra} D(G) \ot{D(P),\iota} \big( W \ot{K} V'_b \big) , \\
		\delta \otimes \mu \otimes w \otimes \ell  &\lto \delta \ast \mu \otimes w \otimes \ell. \nonumber
	\end{align}
	It follows from \Cref{Thm - Comparision between canonical Frechet topology and topological tensor product} applied to $M \defeq U(\Fg) \ot{U(\Fp)} W$ and $V$ that the locally convex $K$-vector space underlying
	\begin{equation*}
		D(G) \ot{D(P),\iota} \big( W \ot{K} V'_b \big)
	\end{equation*}
	is Fr\'echet and thus in particular complete.
	Consequently, we obtain the sought pairing from \eqref{Eq - Generalized Verma module over D(g,P) tensored up} and the $G$-equivariant duality pairing of \Cref{Prop - Module description for induction}
	\begin{align*}
		\langle \blank , \!\blank \rangle \colon  D(G) \ot{D(P),\iota} \big(W \ot{K} V'_b \big) \times \Ind^{\la,G}_{P} \big(W' \ot{K} V \big) &\lra K , \\
		( \delta \otimes w \otimes \ell , f ) &\lto (w\otimes \ell)\big( \delta(f) \big) .
	\end{align*}
\end{proof}

\begin{theorem}\label{Thm - Dual of Orlik-Strauch functors}
	Let $M \in \CO^\Fp_\alg$ and let $V$ be a strongly admissible smooth $P$-representation.
	Then there is a natural topological isomorphisms of $D(G)$-modules 
	\begin{equation*}
		\CF^G_P (M,V)'_b \cong D(G) \ot{D(\Fg,P),\iota} \big(M\ot{K} V' \big)  .
	\end{equation*}
\end{theorem}

When $V$ is the trivial representation, this description of the dual of $\CF^G_P(M,K)$ already is the content of \cite[Prop.\ 3.7]{OrlikStrauch15JordanHoelderSerLocAnRep}.
The above theorem implies that the $p$-adic Orlik--Strauch functor $\CF_P^G$ factors over $\dot{\CF}_P^G$ as follows
\begin{equation*}
	\begin{tikzcd}
		\CO_\alg^\Fp \times \Rep_K^{\sm, \mathrm{s\text{-}adm}} (P) \ar[rrr, "\CF_P^G"] \ar[dr, "\CC \times (\blank)'_b"', end anchor= 164 ] &[-30pt] &[10pt] &[-30pt] \Rep_K^{\la,\adm}(G) \\
		&\CM_{D(\Fg,P)}^\mathrm{nF} \times \CM_{D(\Fg,P)}^\mathrm{nF} \ar[r, "\blank \, \widehat{\otimes}_{K} \blank"] & \CM_{D(\Fg,P)}^\mathrm{nF} \ar[ur, "\dot{\CF}_P^G"'] & .
	\end{tikzcd}
\end{equation*}

\begin{proof}[Proof of \Cref{Thm - Dual of Orlik-Strauch functors}]
	We may find a finite-dimensional algebraic $\bP_K$-representation $W$ and a short exact sequence
	\begin{equation*}
		0 \lra \Fd \lra U(\Fg) \ot{U(\Fp)} W \lra M \lra 0 
	\end{equation*}
	of $U(\Fg)$-modules.
	Prescribing the trivial $\Fg$-action on $V'$ and taking the tensor product of this sequence with $V'$ over $K$, we arrive at a short exact sequence of $D(\Fg,P)$-modules.
	Its constituents are finitely presented over $D(\Fg,P_0)$ \cite[Prop.\ 4.1.5]{AgrawalStrauch22FromCatOLocAnRep}.
	Hence it follows from \cite[Cor.\ 7.8.7]{AgrawalStrauch22FromCatOLocAnRep} that the resulting
	\begin{align*}
		0 \lra D(G_0) \ot{D(\Fg,P_0)} \big( \Fd \ot{K} V') 
		\lra &{\hspace{2pt}} D(G_0) \ot{D(\Fg,P_0)} \Big( \big( U(\Fg) \ot{U(\Fp)} W \big) \ot{K} V' \Big) \\
		&\qquad\lra D(G_0) \ot{D(\Fg,P_0)} \big( M \ot{K} V' \big) \lra 0 
	\end{align*}
	is a short exact sequence of coadmissible $D(G_0)$-modules.
	By \Cref{Thm - Comparision between canonical Frechet topology and topological tensor product} it is a short strictly exact sequence of jointly continuous $D(G_0)$-modules.
	In view of \Cref{Lemma - Tensoring up from composite algebra} it thus yields a topological isomorphism of separately continuous $D(G)$-modules
	\begin{align}\label{Eq - Quotient for tensored up generalized Verma module}
		D(G) \ot{D(\Fg,P),\iota} &\big( M \ot{K} V'_b \big)\nonumber \\ 
			&\cong D(G) \ot{D(\Fg,P),\iota} \Big( \big(U(\Fg) \ot{U(\Fp)} W \big) \ot{K} V'_b \Big)
		\Big/ D(G) \ot{D(\Fg,P),\iota} \big( \Fd \ot{K} V'_b \big) 
	\end{align}

	\begin{lemma}\label{Lemma - Compatibility of duality pairings}
			For $ \Fz \in U(\Fg)\ot{U(\Fp)} W$, $g\in G$, $\ell \in V'_b$ and $f\in \Ind^{\la, G}_P \big( W'\ot{K} V \big)$, we have for the duality pairing of \Cref{Lemma - Some duality pairing}
			\begin{equation*}
				\langle \delta_{g} \otimes \Fz \otimes \ell , f \rangle = \ell \big( \langle \Fz, f \rangle_{C^\la(G,V)} (g) \big) .
			\end{equation*}
	\end{lemma}
	\begin{proof}
		The statement follows once we show that, for all $w\in W$, $\ell \in V'_b$, $g \in G$, $\mu \in U(\Fg)$,
		\begin{equation*}
			(w\otimes \ell) \big( I(f)(\delta_g \ast \mu) \big) = \ell \Big( I\big(x \mto \ev_w(f(gx)) \big) (\mu) \Big) 
			\quad\text{, for all $f \in C^\la \big(G, W' \ot{K} V \big)$.}
		\end{equation*}
		But this is a consequence of the commutativity of the following diagram:
		\begin{equation*}
			\begin{tikzcd}
				C^\la \big(G, W' \ot{K} V \big) \ar[r, "I_{W'\otimes_{K}V}"] \ar[d, "\rho_l(g^{-1})"] &[15pt] \CL_b \big(D(G), W'\ot{K}V \big) \ar[rd, "\ev_{\delta_g \ast \mu}"] \ar[d, "\eta_g"] & & \\
				C^\la \big(G, W' \ot{K} V \big) \ar[r, "I_{W'\otimes_{K}V}"] \ar[d, "(\ev_{w})_\ast"] &\CL_b \big(D(G), W'\ot{K}V \big) \ar[r, "\ev_{\mu}"] \ar[d, "(\ev_{w})_\ast"]& W'\ot{K} V \ar[r, "w \otimes \ell"] \ar[d, "\ev_w"] &K \\
				C^\la(G,V) \ar[r, "I_V"] & \CL_b (D(G), V) \ar[r, "\ev_\mu"] & V \ar[ru, "\ell"] . &
			\end{tikzcd}
		\end{equation*}
		Here $\rho_l$ denotes the left regular $G$-action on $C^\la \big(G, W'\ot{K} V \big)$, and $\eta_g$ the map that sends $\varphi \in \CL_b \big(D(G), W'\ot{K}V \big)$ to the homomorphism $\varphi( \delta_g \ast \blank)$.
	\end{proof}

	Using this lemma, we now adapt the proof of \cite[Prop.\ 3.3]{OrlikStrauch15JordanHoelderSerLocAnRep}.
	By \cite[Thm.\ 6.3, Lemma 3.6]{SchneiderTeitelbaum03AlgpAdicDistAdmRep} the admissible subrepresentations of $\Ind^{\la, G}_P( W'\ot{K} V )$ correspond to closed $D(G)$-submodules of its continuous dual \eqref{Eq - Generalized Verma module over D(g,P) tensored up}.
	Explicitly, such a submodule $\mathfrak{D}$ corresponds to the subrepresentation
	\begin{equation*}
		\Big\{ f\in \Ind^{\la, G}_P \big(W'\ot{K} V \big) \,\Big\vert\, \forall \FZ \in \mathfrak{D}: \langle \FZ, f \rangle = 0 \Big\} ,
	\end{equation*}
	and $\langle \blank ,\!\blank \rangle$ puts this subrepresentation into duality with the quotient
	\begin{equation*}
		D(G) \ot{D(\Fg,P),\iota} \Big( \big( U(\Fg) \ot{U(\Fp)} W \big) \ot{K} V'_b \Big)  \Big/ \mathfrak{D} .
	\end{equation*}
	We claim that, for $f\in \Ind^{\la, G}_P \big(W'\ot{K} V \big)$, the following are equivalent:
	\begin{altenumeratelevel2}
		\item
		$\forall \FZ \in D(G) \ot{D(\Fg,P),\iota} \big( \Fd \ot{K} V'_b \big) : \langle \FZ , f\rangle = 0 $,
		\item
		$\forall \Fz \in \Fd: \langle \Fz, f \rangle_{C^\la(G,V)} = 0 $.
	\end{altenumeratelevel2}
	Indeed, assume (1) and let $\Fz \in \Fd$, $g\in G$ and $\ell \in V'_b$. 
	Then it follows from \Cref{Lemma - Compatibility of duality pairings} that
	\begin{align*}
		 \ell \big( \langle \Fz, f \rangle_{C^\la(G,V)} (g) \big)  = \langle \delta_{g} \otimes \Fz \otimes \ell, f \rangle 
		=0.
	\end{align*}
	Since $\ell \in V'_b$ and $g\in G$ were arbitrary, we conclude that $\langle \Fz, f \rangle_{C^\la(G,V)}$ is the zero function.
	The other implication follows directly from \Cref{Lemma - Compatibility of duality pairings} and density of Dirac distributions in $D(G)$.
	The above claim implies that the subrepresentation $\Ind^{\la, G}_P \big(W'\ot{K} V \big)^\Fd $ corresponds to the submodule $\mathfrak{D} \defeq D(G) \ot{D(\Fg,P),\iota} \big( \Fd \ot{K} V'_b \big) $.
	Hence the theorem follows from \eqref{Eq - Quotient for tensored up generalized Verma module}.
\end{proof}

\AtNextBibliography{\small}
\printbibliography

\end{document}